\documentclass[12pt]{amsart}
\usepackage{mathrsfs,amscd,amssymb,amsthm,amsmath,array,wasysym,textcomp,comment,mathtools}
\usepackage[matrix,arrow,curve]{xy}
\usepackage[english]{babel}
\usepackage[utf8]{inputenc}
\usepackage[margin=2cm]{geometry}

\newtheorem{theorem}[equation]{Theorem}

\newtheorem{proposition}[equation]{Proposition}
\newtheorem{lemma}[equation]{Lemma}
\newtheorem{corollary}[equation]{Corollary}
\newtheorem{conjecture}[equation]{Conjecture}

\newtheorem{question}[equation]{Question}
\newtheorem{exercise}[equation]{Exercise}

\theoremstyle{definition}
\newtheorem{example}[equation]{Example}

\theoremstyle{remark}
\newtheorem{remark}[equation]{Remark}

\makeatletter\@addtoreset{equation}{section} \makeatother

\newcommand{\mumu}{\boldsymbol{\mu}}

\swapnumbers
\newtheoremstyle{dotless}{}{}{\rm}{}{\sc}{}{ }{}

\theoremstyle{dotless}
\newtheorem{case}[subsection]{Case:}

\newcommand{\Q}{\mathbb{Q}}

\renewcommand{\SS}{\mathfrak{S}}
\newcommand{\M}{\mathcal{M}}

\newcommand{\bP}{\mathbb{P}}

\newcommand{\bQ}{\mathbb{Q}}

\newcommand{\bT}{\mathbb{T}}

\newcommand{\cC}{\mathcal{C}}
\newcommand{\cO}{\mathcal{O}}
\newcommand{\cL}{\mathcal{L}}

\newcommand{\cM}{\mathcal{M}}

\newcommand{\cJ}{\mathcal{J}}

\newcommand{\Bs}{\mathrm{Bs}}

\DeclareMathOperator{\mult}{mult}

\title{Birational rigidity and alpha invariants of Fano varieties}

\author{Ivan Cheltsov, Arman Sarikyan, Ziquan Zhuang}

\dedicatory{To Slava Shokurov on the~occasion of his 70th birthday.}

\address{\emph{Ivan Cheltsov}
\newline
\textnormal{University of Edinburgh,  Edinburgh, Scotland.}
\newline
\textnormal{\texttt{I.Cheltsov@ed.ac.uk}}}

\address{\emph{Arman Sarikyan}
\newline
\textnormal{University of Edinburgh,  Edinburgh, Scotland.}
\newline
\textnormal{\texttt{ar.sarikyan@gmail.com}}}

\address{\emph{Ziquan Zhuang}
\newline
\textnormal{Johns Hopkins University, Baltimore, MD, USA.}
\newline
\textnormal{\texttt{zzhuang@jhu.edu}}}

\begin{document}

\begin{abstract}
We prove that for every $\epsilon>0$, there is a~birationally super-rigid Fano variety $X$
such that $\frac{1}{2}\leqslant\alpha(X)\leqslant \frac{1}{2}+\epsilon$.
Also we show that for every $\epsilon>0$, there is a~Fano variety $X$ and a~finite subgroup $G\subset\mathrm{Aut}(X)$
such that $X$ is $G$-birationally super-rigid, and $\alpha_G(X)<\epsilon$.
\end{abstract}

\maketitle

\tableofcontents

Throughout this paper, we assume that all varieties are projective, normal and defined over $\mathbb{C}$.

\section{Introduction}
\label{section:intro}

Let $X$ be a~Fano variety with terminal singularities.
If $\mathrm{rk}\,\mathrm{Cl}(X)=1$, then $X$ is a~Mori fibre~space.
In this case, we say that $X$ is \emph{birationally rigid} if $X$ is not birational to other Mori fibre spaces~\cite{Ch05}.
Similarly, we say that $X$ is \emph{birationally super-rigid} if it is birationally rigid and
$\mathrm{Bir}(X)=\mathrm{Aut}(X)$.
Examples of birationally super-rigid smooth Fano varieties include
\begin{itemize}
\item smooth hypersurfaces in $\mathbb{P}^{n+1}$ of degree $n+1\geqslant 4$ \cite{Cheltsov2000,EinFernexMustata,Fernex,IsMa71,Kollar,Pukhlikov87,Pukhlikov97,Zhuang};
\item smooth weighted hypersuraces in $\mathbb{P}(1^{n+1},n)$ of degree $2n\geqslant 6$ \cite{Pukhlikov89}.
\end{itemize}
Note that these examples of smooth Fano varieties are known to be K-stable \cite{AbbanZhuang2021,Cheltsov2001,CheltsovPark,CPW,Dervan2,Fujita}.
One can prove this by using Tian's criterion. Namely, recall from \cite{OdakaSano,Ti87} that $X$ is K-stable~if
$$
\alpha(X)>\frac{\mathrm{dim}(X)}{\mathrm{dim}(X)+1},
$$
where $\alpha(X)$ is the~$\alpha$-invariant of $X$ that can be defined as follows:
$$
\alpha(X)=\mathrm{sup}\left\{\lambda\in\mathbb{Q}\ \left|\ \aligned
&\text{the log pair}\ \left(X, \lambda D\right)\ \text{is log canonical}\\
&\text{for any effective $\mathbb{Q}$-divisor}\ D\sim_{\mathbb{Q}} -K_{X}\\
\endaligned\right.\right\}.
$$
If $X$ is smooth, then $X$ is also K-stable in the~case when $\alpha(X)=\frac{\mathrm{dim}(X)}{\mathrm{dim}(X)+1}$ and~\mbox{$\mathrm{dim}(X)\geqslant 2$}~\cite{Fujita}.
On the~other hand, if $X$ is a~smooth hypersurface in $\mathbb{P}^{n+1}$ of degree $n+1$, then \cite{Cheltsov2001,CheltsovPark} gives
$$
\alpha(X)\geqslant\frac{n}{n+1}=\frac{\mathrm{dim}(X)}{\mathrm{dim}(X)+1}.
$$
Similarly, if $X$ is smooth hypersuraces in $\mathbb{P}(1^{n+1},n)$ of degree $2n\geqslant 2$,
then \cite{CPW} gives
$$
\alpha(X)\geqslant\frac{2n-1}{2n}>\frac{n}{n+1}=\frac{\mathrm{dim}(X)}{\mathrm{dim}(X)+1}.
$$
This shows that all smooth hypersurfaces in $\mathbb{P}^{n+1}$ of degree $n+1\geqslant 3$
and all smooth weighted hypersuraces in $\mathbb{P}(1^{n+1},n)$ of degree $2n\geqslant 4$ are K-stable.
This gives an evidence for

\begin{conjecture}[{\cite{KimOkadaWon}}]
\label{conjecture:Odaka-Okada}
Let $X$ be a~Fano variety with terminal singularities such that $\mathrm{rk}\,\mathrm{Cl}(X)=1$.
Suppose that $X$ is birationally rigid. Then $X$ is K-stable.
\end{conjecture}

This conjecture has been already verified for many Fano varieties \cite{Ch08b,Ch09a,Ch09b,KimOkadaWon,CP,CPR,Zhuang,Suzuki},
but it is still open in full generality (cf. \cite{OdakaOkada}). On the~other hand, we have the~following result:

\begin{theorem}[\cite{StibitzZhuang}]
\label{theorem:Stibizt-Zhuang}
Let $X$ be a~Fano variety with terminal singularities such that $\mathrm{rk}\,\mathrm{Cl}(X)=1$.
Suppose that $X$ is birationally super-rigid and $\alpha(X)\geqslant\frac{1}{2}$. Then $X$ is K-stable.
\end{theorem}

This naturally leads to the~question:

\begin{question}[{\cite{StibitzZhuang}}]
\label{question:Stibizt-Zhuang}
Is it true that $\alpha(X)\geqslant\frac{1}{2}$ for any  birationally super-rigid Fano variety $X$?
\end{question}

In this paper we show that the~bound $\frac{1}{2}$ is optimal by proving the~following theorem.

\begin{theorem}
\label{theorem:alpha}
For every $\epsilon>0$, there exists a~singular Fano variety $X$ with terminal singularities such that $\mathrm{rk}\,\mathrm{Cl}(X)=1$,
the variety $X$ is birationally super-rigid, and
$$
\frac{1}{2}\leqslant\alpha(X)\leqslant \frac{1}{2}+\epsilon.
$$
\end{theorem}

We also answer a~natural equivariant version of Question~\ref{question:Stibizt-Zhuang}, which can be stated as follows.
Suppose that $\mathrm{rk}\,\mathrm{Cl}^G(X)=1$ for a~finite subgroup $G\subset\mathrm{Aut}(X)$,
so that $X$ is a~$G$-Mori~fibre~space.
Then $X$ is \emph{$G$-birationally rigid} if it is not $G$-birational to other $G$-Mori fibre spaces~\mbox{\cite[\S~3.1.1]{CheltsovShramov}}.
Similarly, the~Fano variety $X$ is said to be $G$-birationally super-rigid if $X$ is $G$-birationally rigid, and $X$ does not have non-biregular $G$-birational selfmaps.
Finally, we let
$$
\alpha_G(X)=\mathrm{sup}\left\{\lambda\in\mathbb{Q}\ \left|\ \aligned
&\text{the pair}\ \left(X, \lambda D\right)\ \text{is log canonical for every}\\
&\text{effective $G$-invariant $\mathbb{Q}$-divisor}\ D\sim_{\mathbb{Q}} -K_{X}\\
\endaligned\right.\right\}.
$$
If $\alpha_G(X)>\frac{\mathrm{dim}(X)}{\mathrm{dim}(X)+1}$, then $X$ is K-polystable by \cite[Corollary~1.3]{Zhuang2020}.

\begin{question}
\label{question:G-invariant}
Is it true that $\alpha_G(X)\geqslant\frac{1}{2}$ for any  $G$-birationally super-rigid Fano variety $X$?
\end{question}

The answer to this question is positive in dimension two:

\begin{exercise}[{\cite{Ch08a,ChSh11,Sakovics}}]
\label{exercise:Sakovich}
If $\mathrm{dim}(X)=2$ and $X$ is $G$-birationally super-rigid, then $\alpha_G(X)\geqslant\frac{2}{3}$.
\end{exercise}

In dimension three we still do not know whether our Question~\ref{question:G-invariant} has a~positive answer or not,
but~many examples suggest that the~answer is probably positive.

\begin{example}[{\cite{ChSh10a,ChSh09b,CheltsovShramov2017}}]
\label{example:Blichfeldt}
Suppose that $X=\mathbb{P}^3$, and let $G$ be any finite subgroup in $\mathrm{Aut}(X)$.
Then $X$ is $G$-birationally super-rigid if and only if the~following four conditions are satisfied
\begin{itemize}
\item[(i)] $X$ does not have $G$-orbits of length $\leqslant 4$;
\item[(ii)] $X$ does not contains $G$-invariant lines;
\item[(iii)] $X$ does not contains $G$-invariant pairs of skew lines;
\item[(iv)] $G$ is not isomorphic to $\mathfrak{A}_5$, $\mathfrak{S}_5$, $\mathrm{PSL}_{2}(\mathbf{F}_7)$, $\mathfrak{A}_6$, $\mumu_2^4\rtimes\mumu_5$ and $\mumu_2^4\rtimes\mathrm{D}_{10}$.
\end{itemize}
Using this criterion and \cite{ChSh11}, we see that $\alpha_G(X)\geqslant\frac{1}{2}$ if $X$ is $G$-birationally super-rigid.
\end{example}

In this paper, we prove that the~answer to Question~\ref{question:G-invariant} is very negative in higher dimensions.

\begin{theorem}
\label{theorem:main}
For every $\epsilon>0$, there is a~smooth Fano variety $X$ and a~finite subgroup~\mbox{$G\subset\mathrm{Aut}(X)$}
such that $\mathrm{rk}\,\mathrm{Pic}^G(X)=1$, the~variety $X$ is $G$-birationally super-rigid, and $\alpha_G(X)<\epsilon$.
\end{theorem}

Let us describe the~structure of this paper.
In Section~\ref{section:Arman}, we prove Theorem~\ref{theorem:alpha}.
In Section~\ref{section:quadric-threefold}, we study  equivariant birational
geometry of a~smooth quadric threefold $Q\subset\mathbb{P}^4$ for the~natural action of the~symmetric group $\mathfrak{S}_5$,
which should be interesting for mathematicians working on finite subgroups of the~space Cremona group (cf. \cite[\S~9]{Tschinkel}).
This example inspired Theorem~\ref{theorem:main}.
In~Section~\ref{section:preliminary}, we present few results used in the~proof of Theorem~\ref{theorem:main},
which is done in Section~\ref{section:the-proof}.

\medskip
\noindent
\textbf{Acknowledgements.}
Ivan Cheltsov was supported by the~EPSRC grant number~EP/V054597/1.
He~worked on this paper during 2 months stay at Institut des Hautes \'{E}tudes Scientifiques (IH\'{E}S).
Ivan would like to thank the~institute for good working conditions.
Ziquan Zhuang was supported by the~NSF Grant DMS-2240926 and a~Clay research fellowship.

\section{The proof of Theorem~\ref{theorem:alpha}}
\label{section:Arman}

We fix a~positive integer $a\geqslant 2$.
Then we let $X$ be a~quasi-smooth well-formed singular weighted hypersurfaces of degree $2a+1$ in $\mathbb{P}(1^{a+2},a)$ that is given by the~following equation:
$$
y^2x_1+f_{2a+1}(x_1,\ldots,x_{a+2})=0,
$$
where each $x_i$ is a~coordinate of weight $1$, $y$ is a~coordinate of weight $a$, and $f_{2a+1}$ is a~general homogeneous polynomial of degree $2a+1$.
Then
\begin{itemize}
\item $X$ is a~Fano variety of dimension $N=a+1$,
\item the~class group of the~variety $X$ is of rank $1$,
\item the~singularities of $X$ consist of one singular point $O_y=(0:\ldots:0:1)$,
which is a~terminal quotient singularity of type $\frac{1}{a}(1,\ldots,1)$.
\end{itemize}
Further, it follows from \cite{KimOkadaWon2} that
$$
\alpha(X)\leqslant\frac{a+1}{2a+1}=\frac{1}{2}+\frac{1}{4a+2}.
$$
In this section, we prove the~following result, which implies Theorem~\ref{theorem:alpha}.

\begin{theorem}
\label{theorem:Arman}
The Fano variety $X$ is birationally super-rigid.
\end{theorem}

This theorem also answers positively \cite[Question 7.2.3]{KimOkadaWon2}.

\begin{remark}
\label{remark:a-2}
If $a=2$, then $X$ is known to be birationally super-rigid \cite{CP,CPR}.
\end{remark}

Let $\pi \colon X\dashrightarrow \mathbb{P}^N$ be the~projection from the~point $O_y$. Then $\pi$ contracts the~following divisor:
$$
D=\big\{x_1=0,f_{2a+1}(x_1,\ldots,x_{a+2})=0\big\}\subset\mathbb{P}(1^{a+2},a).
$$
Further, one has the~following diagram:
$$
\xymatrix{
&\widetilde{X}\ar[d]_f\ar[rr]^\nu \ar[drr]^g&& U\ar[d]^\theta \\
&X\ar@{-->}[rr]^\pi && \mathbb{P}^N}
$$
where $f$ is the~weighted blow-up of the~point $O_y$ with weights $(1,\ldots,1)$,
the map $g$ is a~morphism,
the variety $U$ is a~hypersurface in $\mathbb{P}(1^{a+2},a+1)$ of degree $2a+2$ that is given by
$$
z^2+x_1f_{2a+1}(x_1,\ldots,x_{a+2})=0,
$$
the~morphism $\nu$ is a~birational morphism that contracts the~strict transform of the~divisor $D$,
and the~morphism $\theta$ is a~double cover that is branched over the~hypersurface $x_1f_{2a+1}(x_1,\ldots,x_{a+2})=0$.
Here, we consider $x_1,\ldots,x_{a+2}$ as coordinates on $\mathbb{P}^N$ and as coordinates of weight $1$ on the~weighted projective space $\mathbb{P}(1^{a+2},a+1)$,
where $z$ is a~coordinate of weight $a+1$.

Now, let us prove Theorem~\ref{theorem:Arman}. Assume the~contrary, i.e. there exist a~birational map
$$
\Phi\colon X \dashrightarrow W
$$
to a~Mori fibre space $W$ that is not isomorphism. Let $\mathcal{M}$ be a~birational transform of a~very ample complete linear system on $W$ via $\Phi$. Let $\lambda \in \mathbb{Q}_{>0}$ be the~positive rational number such that
$$
K_{X}+\lambda \mathcal{M}\sim_{\mathbb{Q}}0.
$$
Then, by the~Noether-Fano inequality \cite{CortiFactoring}, the~singularities of the~pair $(X,\lambda\M)$ are not canonical.
Let $Z$ be a~center of non-canonical singularities of the~log pair $(X,\lambda\M)$.

Now, let $E$ be the~$f$-exceptional divisor,
and let $\widetilde{\mathcal{M}}$ be the~strict transform of the~mobile linear system $\mathcal{M}$ on the~variety $\widetilde{X}$.
Then $E\cong\mathbb{P}^{a}$, and
\begin{align*}
K_{\widetilde{X}}&\sim_{\mathbb{Q}}f^*\big(K_{X}\big)+\tfrac{1}{a}E,\\
\lambda \widetilde{\mathcal{M}}&\sim_{\mathbb{Q}} f^*\big(\lambda\mathcal{M}\big)-\mu E,
\end{align*}
for some $\mu \in \mathbb{Q}_{\geqslant0}$. Therefore, we have
$$
K_{\widetilde{X}}+\lambda \widetilde{\mathcal{M}}\sim_{\mathbb{Q}}f^*(K_{X}+\lambda \mathcal{M})+\Big(\frac{1}{a}-\mu\Big)E.
$$
Thus, if $\mu>\frac{1}{a}$, then $O_y$ is a~center of non-canonical singularities of the~ log pair $(X,\lambda\M)$.

\begin{lemma}[{cf. \cite{Kawamata96} for $a=2$}]
\label{lemma:Kawamatalike_inequality}
Suppose that $O_y\in Z$. Then $\mu>\frac{1}{a}$.
\end{lemma}

\begin{proof}
Suppose that $\mu\leqslant \frac{1}{a}$. Let us seek for a~contradiction.
\begin{case}{\sc $Z\neq O_y$.}
Let $\widetilde{Z}$ be the~strict transform of $Z$ via~$f$.
Then $\mult_{\widetilde{Z}}(\widetilde{\mathcal{M}})>\frac{1}{\lambda}$ and hence,
\begin{equation}
\label{equation:mult-big}
\mult_{P}\big(\widetilde{\mathcal{M}}\big|_E\big)>\frac{1}{\lambda}
\end{equation}
for any point $P\in \widetilde{Z}\cap E$. Notice that
$$
\lambda \widetilde{\mathcal{M}}\big|_E\sim_{\mathbb{Q}}-\mu E\big|_E\sim_{\mathbb{Q}} a\mu H,
$$
where $H$ is a~hyperplane in $E\cong\mathbb{P}^a$.
Since $a\mu\leqslant 1$, this contradicts to \eqref{equation:mult-big}.
\end{case}
\begin{case}{\sc $Z=O_y$.}
We write
$$
K_{\widetilde{X}}+\lambda \widetilde{\mathcal{M}}+\Big(\mu-\frac{1}{a}\Big)E\sim_{\mathbb{Q}}f^*\big(K_X+\lambda \mathcal{M}\big).
$$
Hence, the~singularities of the~log pair $(\widetilde{X}, \lambda\widetilde{\mathcal{M}} +(\mu-\tfrac{1}{a})E)$ are not canonical at some point $P\in E$.
Then the~singularities of the~log pair $(\widetilde{X}, \lambda\widetilde{\mathcal{M}})$ are also not canonical at $P$, so that
$$
\mathrm{mult}_P\big(\widetilde{\mathcal{M}}\big)>\frac{1}{\lambda}.
$$
Now, we argue as in the~previous case to obtain a~contradiction.
\end{case}
\end{proof}

One the~other hand, we have

\begin{lemma}
\label{lemma:singular_points}
One has $\mu\leqslant\frac{1}{a}$.
\end{lemma}

\begin{proof}
One has  $g^*(\mathcal{O}_{\mathbb{P}^{N}}(1))\sim_{\mathbb{Q}}f^*(-K_X)-\frac{1}{a}E$. Then
$$
\Big(f^*(-K_X)-\frac{1}{a}E\Big)\cdot C=0,
$$
for any curve $C$ contracted by $g$.
Thus, if $\mu>\frac{1}{a}$, then
$$
\widetilde{M}\cdot C=\frac{1}{\lambda}\big(f^*(-K_X)-\mu E\big)<0
$$
for a~general divisor $\widetilde{M}\in\widetilde{\mathcal{M}}$.
This is a~contradiction, because the~linear system $\widetilde{\mathcal{M}}$ is mobile, and the~curves contracted by $g$ span a~divisor in $\widetilde{X}$ --- the~proper transform of the~divisor $D$.
\end{proof}

\begin{corollary}
\label{corollary:O-not-in-Z}
One has $O_y\notin Z$.
\end{corollary}

Thus, we see that $Z$ is contained in the~smooth locus of the~variety $X$.

\begin{lemma}
\label{lemma:smooth-points}
One has $\mathrm{dim}(Z)=a-1$.
\end{lemma}

\begin{proof}
Suppose that $\mathrm{dim}(Z)<a-1$.
Let $M_1$ and $M_2$ be sufficiently general divisors in $\mathcal{M}$,
and let $P$ be a~sufficiently general point in $Z$.
Then
$$
\left(M_1\cdot M_2\right)_P>\frac{4}{\lambda^2}
$$
by \cite{Pukhlikov97} or \cite[Corollary~3.4]{corti}.
Let $\mathcal{L}$ be the~linear subsystem in $|-K_X|$ consisting of all divisors that pass through the~point $P$,
and let $H_1,\ldots, H_{N-2}$ be sufficiently general divisors in the~system~$\mathcal{L}$.
If $P\not\in D$, then the~base locus of $\mathcal{L}$ does not contain curves, which gives
$$
\frac{2a+1}{a\lambda^2}=M_1\cdot M_2 \cdot H_1\cdot\ldots\cdot H_{N-2}\geqslant \left( M_1\cdot M_2\right)_P>\frac{4}{\lambda^2},
$$
which is a~contradiction. Thus, we see that $P\in D$.

Let $L\subset D$ be the~curve containing $P$ that is contracted by $\pi$.
Then $L$ is the~only curve contained in the~base locus of the~linear system $\mathcal{L}$.
After a~linear change of coordinates,
we can assume that
$$
P=(0:0:1:0:\ldots:0:1)
$$
and $H_i=X \cap \{x_{i+3}=0\}$ for $i=1,\ldots, N-2$. Consider the~surface $S$ defined as
$$
S=\bigcap_{i=1}^{N-2} H_i.
$$
We can identify $S$ with a~surface in $\mathbb{P}(1,1,1,a)$ given by
$$
y^2x_1+f_{2a+1}(x_1,x_2,x_3,0,\ldots,0)=0.
$$
Then $L=S\cap \{x_1=x_2=0\}$. Let $\mathcal{M}_S=\mathcal{M}|_S$. Then
$\lambda\mathcal{M}_S=mL+\lambda\Delta$ for some non-negative rational number $m\in \mathbb{Q}_{\geqslant 0}$
and some mobile linear system $\Delta$ on the~surface $S$.
Moreover, applying the~inversion of adjunction \cite[Theorem~5.50]{KoMo98},
we see that $(S,\lambda \mathcal{M}_S)$ is not log canonical at $P$.

Let $H$ be a~general curve in $|\mathcal{O}_S(1)|$,
and let $H_L$ be a~general curve in $|\mathcal{O}_S(1)|$ that contains $L$. Then $H\cdot L=\frac{1}{a}$ and
$$
S\cap H_L= L+R,
$$
where $R$ is a~curve in $S$ such that $L\not\subset\mathrm{Supp}(R)$.
One can check that $L\cdot R=2$ and $H\cdot R=2$. Thus, using $(L+R)\cdot L=H\cdot L=\frac{1}{a}$, we get
$$
L^2=-2+\frac{1}{a},
$$	
which can also be shown using the~subadjunction formula on $S$.

Now, using Corti's inequality \cite[Theorem 3.1]{corti}, we get
\begin{multline*}
4(1-m)<\lambda^2\big(\Delta_1\cdot\Delta_2\big)_P\leqslant \lambda^2\Delta_1\cdot\Delta_2=\\
=\big(H-mL\big)^2 =H^2-2m H\cdot L+m^2 L^2=\frac{2a+1}{a}-\frac{2m}{a}+m^2\left(-2+\frac{1}{a} \right),	
\end{multline*}
which gives
\begin{align*}
0>\frac{(2a-1)(m-1)^2}{a}.
\end{align*}
This is a~contradiction, since $a\geqslant 2$.
\end{proof}

Therefore, we see that $\mathrm{dim}(Z)=\mathrm{dim}(X)-2$. Then
$$
\mathrm{mult}_{Z}\big(\M\big)>\frac{1}{\lambda}.
$$
Let $M_1$ and $M_2$ be general divisors in $\M$. Then
$$
\frac{3}{\lambda^2}>\frac{2a+1}{\lambda^2 a}=\big(-K_X\big)^{N-2}\cdot M_1\cdot M_2
\geqslant\mult^2_Z(\M)\big(-K_X\big)^{N-2}\cdot Z>\frac{1}{\lambda^2} \big(-K_X\big)^{N-2}\cdot Z,
$$
so that $(-K_X)^{N-2}\cdot Z\in\{1,2\}$.

Now, let $H_1,\ldots, H_{N-2}$ be general divisors in $|-K_X|$.
After a~linear change of coordinate system one can assume that $H_i=X \cap \{x_{i+4}=0\}$ for $i=1,\ldots, N-3$.
Let $V$ be the~threefold defined as
$$
V=\bigcap _{i=1}^{N-3} H_i.
$$
Then we can identify $V$ with the~hypersurface in $\mathbb{P}(1^4,a)$ given by
$$
y^2x_1+f_{2a+1}(x_1,\ldots,x_4,0,\ldots,0)=0.
$$
Let $C=V\cap Z$, $\mathcal{M}_{V}= \mathcal{M}|_{V}$,
and let $H$ be a~general surface in $|\mathcal{O}_V(1)|$.
Then
\begin{itemize}
\item $C$ is an irreducible curve such that $C\cdot H\in\{1,2\}$,
\item $C$ is contained in the~smooth locus of the~hypersurface $V$,
\item $C$ is a~center of non-canonical singularities of the~log pair $(V,\lambda \mathcal{M}_{V})$.
\end{itemize}
We set $\mu=\lambda\mathrm{mult}_C(\mathcal{M}_{V})$. Then $\mu>1$.

\begin{lemma}
\label{lemma:C-H-1}
One has $C\cdot H\ne 1$.
\end{lemma}

\begin{proof}
Suppose that $C\cdot H=1$.
We can choose coordinates on $\mathbb{P}(1^4,a)$ such that
$$
C=\big\{x_2=0,x_3=0,y+F(x_1,\ldots, x_4)=0\big\}\subset\mathbb{P}(1^4,a),
$$
where $F(x_1,\ldots x_4)$ is a~homogeneous polynomial of degree $a$. Note that $C\cong\mathbb{P}^1$.

Now, we let $\beta\colon \widetilde{V}\to V$ be the~blow-up of the~curve $C$, and let $E$ be the~$\beta$-exceptional divisor.
We claim that $E^3=a-1$. Indeed, let $\widetilde{S}_1$, $\widetilde{S}_2$, $\widetilde{S}_3$
be the~strict transforms on $\widetilde{V}$ of the~surfaces that are cut out on $V$ by the~equations $y+F(x_1,\ldots, x_4)=0$, $x_2=0$, $x_3=0$, respectively.
Then
$$
0=\widetilde{S}_1\cdot \widetilde{S}_2\cdot\widetilde{S}_3=(a\beta^*(H)-E)\cdot(\beta^*(H)-E)^2= aH^3+(a+2)\beta^*(H)\cdot E^2 - E^3= a-1-E^3,
$$
which gives $E^3=a-1$ as claimed.

Let $\mathcal{M}_{\widetilde{V}}$ be the~strict transform of the~linear system $\mathcal{M}_{V}$ on the~threefold $\widetilde{V}$.
Then
$$
\lambda\mathcal{M}_{\widetilde{V}}\sim_{\mathbb{Q}}\beta^*(H)-\mu E,
$$
One the~other hand, since $\mathcal{M}_{\widetilde{V}}$ is mobile and $a\beta^*(H)-E$ is nef, we get
$$
0\leqslant\big(a\beta^*(H)-E\big)\cdot\big(\beta^*(H)-\mu E\big)^2=aH^3+(2\mu+a\mu ^2)\beta^*(H)\cdot E^2-\mu^2E^3=(\mu -1)^2-2a(\mu^2-1 )<0,
$$
which is a~contradiction.
\end{proof}

Thus, we see that $C\cdot H=2$.
Then we can change coordinates on $\mathbb{P}(1^4,a)$ such that
\begin{itemize}
\item[($\mathrm{A}$)] either
$$
C=\big\{x_4=0,x_1x_2+x_3^2=0,y+F_a(x_1,\ldots, x_4)=0\big\}\subset\mathbb{P}(1^4,a)
$$
for some homogeneous polynomial $F_a(x_1,\ldots, x_4)$ of degree $a$,
\item [($\mathrm{B}$)] or
$$
C=\big\{x_2=0,x_3=0,y^2+F_{2a}(x_1,\ldots, x_4)=0\big\}\subset\mathbb{P}(1^4,a)
$$
for some homogeneous polynomial $F_{2a}(x_1,\ldots, x_4)$ of degree $2a$.
\end{itemize}
In case ($\mathrm{A}$), we have $C\cong\mathbb{P}^1$. In case ($\mathrm{B}$), the~curve $C$ may have singularities.

In both cases, let $\beta\colon \widetilde{V}\to V$ be the~blow-up of the~curve $C$, and let $E$ be the~$\beta$-exceptional divisor.
Then, arguing as in the~proof of Lemma~\ref{lemma:C-H-1}, we get
$$
E^3=\left\{\aligned
&2a-4\ \text{in case ($\mathrm{A}$)},\\
&-2\ \text{in case ($\mathrm{B}$)}.
\endaligned
\right.
$$
Let $\mathcal{M}_{\widetilde{V}}$ be the~strict transform of the~linear system $\mathcal{M}_{V}$ on the~threefold $\widetilde{V}$.
Then
$$
\lambda\mathcal{M}_{\widetilde{V}}\sim_{\mathbb{Q}}\beta^*(H)-\mu E.
$$
Moreover, in case ($\mathrm{A}$), the~divisor $a\beta^*(H)-E$ is nef, so that
$$
\big(a\beta^*(H)-E\big)\cdot\big(\beta^*(H)-\mu E\big)^2\geqslant 0
$$
because $\mathcal{M}_{\widetilde{V}}$ is mobile. But
$$
\big(a\beta^*(H)-E\big)\cdot\big(\beta^*(H)-\mu E\big)^2=aH^3+(2\mu+a\mu ^2)\beta^*(H)\cdot E^2 -\mu^2E^3=(2 \mu -1)^2-2a(2\mu^2-1)<0,
$$
because $\mu>1$. Likewise, in case ($\mathrm{B}$), the~divisor $2a\beta^*(H)-E$ is nef, which gives
\begin{multline*}
0\leqslant\big(2a\beta^*(H)-E\big)\cdot\big(\beta^*(H)-\mu E\big)^2=2aH^3+(2\mu +2a\mu ^2)\beta^*(H)\cdot E^2 -\mu ^2E^3=\\
=-2 (\mu -1) (1 -\mu + 2 a~(1 + \mu))<0.
\end{multline*}
Thus, we get a~contradiction in both cases ($\mathrm{A}$) and ($\mathrm{B}$). This completes the~proof of Theorem~\ref{theorem:Arman}.

\section{$\mathfrak{S}_5$-invariant quadric threefold}
\label{section:quadric-threefold}

Let $Q$ be a~smooth quadric hypersurface in $\mathbb{P}^4$.
We can choose coordinates $x_0,x_1,x_2,x_3,x_4$ on the~projective space $\mathbb{P}^4$ such that $Q$ is given by the~following equation:
$$
\sum_{i=0}^4x_i^2=0.
$$
In particular, we see that $Q$ is faithfully acted on by the~symmetric group $\mathfrak{S}_5$,
which permutes the~coordinates $x_0,x_1,x_2,x_3,x_4$.
Then $\alpha_{\mathfrak{S}_5}(Q)\leqslant\frac{1}{3}$,
because $\mathfrak{S}_5$ leaves invariant the~hyperplane sections of the~quadric $Q$
that is cut out by $x_0+x_1+x_2+x_3+x_4=0$.
In fact, arguing as in \cite{ChSh11}, one can show that $\alpha_{\mathfrak{S}_5}(Q)=\frac{1}{3}$.

Keeping in mind the~results obtained in \cite{ChSh09b},
one can expect that $Q$ is $\mathfrak{S}_5$-birationally rigid.
However, this is not the~case --- the~quadric hypersurface $Q$ contains two $\mathfrak{S}_5$-orbits of length $5$,
and each of them leads to a~$G$-birational transformation of the~quadric into other $\mathfrak{S}_5$-Mori fibre~space.
Namely, let $\Sigma_5$ be a~$\mathfrak{S}_5$-orbit of length $5$ in $X$,
and let $\pi\colon X\rightarrow Q$ be the~blow up~of~this $\mathfrak{S}_5$-orbit.
Then there exists the~following $\mathfrak{S}_5$-equivariant commutative diagram:
$$
\xymatrix{
&U\ar@{-->}[rr]^{\zeta}\ar@{->}[ld]_{\pi}&&W\ar@{->}[rd]^{\phi}&\\
Q\ar@{-->}[rrrr]^{\chi} &&  && Y}
$$
where $\zeta$ is a~small birational map that flops the~proper transforms of $10$ conics that contain three points in $\Sigma_5$,
$\phi$ is a~birational morphism that contracts the~proper transforms of $5$ hyperplane sections of the~quadric $Q$ that pass through four points in $\Sigma_5$,
and $Y$ is a~cubic threefold~in~$\mathbb{P}^4$~such that it has $5$ isolated ordinary double points and $\mathrm{rk}\,\mathrm{Cl}(Y)=1$.
Since $Y$ is a~$\mathfrak{S}_5$-Mori fibre~space,
we see that  $Q$ is not \mbox{$\mathfrak{S}_5$-birationally} rigid.
Note that the~cubic threefold $Y$ is given in $\mathbb{P}^4$ by
$$
x_0x_1x_2+x_0x_1x_3+x_0x_1x_4+x_0x_2x_3+x_0x_2x_4+x_0x_3x_4+x_1x_2x_3+x_1x_2x_4+x_1x_3x_4+x_2x_3x_4=0.
$$
This is not difficult to prove \cite{Avilov2016,Avilov2018}.

The goal of this section is to prove the~following result.

\begin{theorem}
\label{theorem:quadric-cubic}
The only $\mathfrak{S}_5$-Mori fiber spaces that are $\mathfrak{S}_5$-birational to~$Q$ are $Q$ and $Y$.
\end{theorem}

Let us prove Theorem~\ref{theorem:quadric-cubic}.
Let $\iota\in\mathrm{Aut}(Q)$ be the~Galois involution of the~double cover $Q\to\mathbb{P}^3$ given by the~projection from the~point $(1:1:1:1:1)$.
Then $\iota$ commutes with the~$\mathfrak{S}_5$-action~on~$Q$.
It is well-known \cite{CheltsovDuboulozKishimoto,CheltsovSarikyan} that Theorem~\ref{theorem:quadric-cubic} follows from the~following technical result:

\begin{theorem}
\label{theorem:quadric-cubic-technical}
Let $\mathcal{M}_Q$ be any non-empty mobile $\mathfrak{S}_5$-invariant linear system on the~quadric~$Q$,
and let $\mathcal{M}_Y$ and $\mathcal{M}_Y^\prime$ be its proper transform on the~cubic threefolds $Y$ via $\chi$ and $\chi\circ\iota$, respectively.
Choose positive rational numbers $\lambda$, $\mu$, $\mu^\prime$  such that
\begin{align*}
\lambda\mathcal{M}_Q&\sim_{\mathbb{Q}} -K_{Q},\\
\mu\mathcal{M}_Y&\sim_{\mathbb{Q}} -K_{Y},\\
\mu^\prime\mathcal{M}_Y^\prime&\sim_{\mathbb{Q}} -K_{Y^\prime}.
\end{align*}
Then one of the~log pair $(Q,\lambda\mathcal{M}_Q)$, $(Y,\mu\mathcal{M}_Y)$ or $(Y^\prime,\mu^\prime\mathcal{M}^\prime_Y)$ has canonical singularities.
\end{theorem}

To prove Theorem~\ref{theorem:quadric-cubic-technical}, let us use all notations and assumptions of this theorem.
We must prove that at least one of the~log pair $(Q,\lambda\mathcal{M}_Q)$, $(Y,\mu\mathcal{M}_Y)$ or $(Y,\mu^\prime\mathcal{M}^\prime_Y)$ has canonical singularities.
Set $\Sigma_5^\prime=\iota(\Sigma_5)$. Then $\Sigma_5^\prime$ is the~second $\mathfrak{S}_5$-orbit in the~quadric $Q$.

\begin{remark}
\label{remark:quadric-sigma-5}
Let $G$ be a~stabilizer in $\mathfrak{S}_5$ of a~point in $P\in\Sigma_5\cup\Sigma_5^\prime$.
Then $G\cong\mathfrak{S}_4$ and its induced linear action on the~Zariski tangent space $T_{P}(Q)$ is an irreducible representation.
\end{remark}

Now using this remark, \cite[Lemma~2.4]{ACPS} and \cite[Theorem~3.10]{corti},
we can easily derive the~required assertion from the~following two propositions,
arguing as in the~proof of \cite[Theorem~1.2]{CheltsovDuboulozKishimoto}.

\begin{proposition}
\label{proposition:quadric-maximal-singularities}
The log pair $(Q,\lambda\mathcal{M}_Q)$ is canonical away from $\Sigma_5\cup\Sigma_5^\prime$.
\end{proposition}

\begin{proposition}
\label{proposition:cubic-maximal-singularities}
The log pairs $(Y,\mu\mathcal{M}_Y)$ and $(Y,\mu^\prime\mathcal{M}^\prime_Y)$ are canonical away from $\mathrm{Sing}(Y)$.
\end{proposition}

In the~remaining part of this section, we will prove Propositions~\ref{proposition:quadric-maximal-singularities} and \ref{proposition:cubic-maximal-singularities}.
For both proofs, we need the~following technical observation, which improves \cite[Lemma~2.2]{ChSh09b}.

\begin{remark}
\label{remark:Ziquan}
Let $X$ be a~variety with terminal singularities,
let $D$ be an effective $\mathbb{Q}$-Cartier divisor on the~variety $X$,
let~$\varphi\colon\widetilde{X}\to X$ be birational morphism such that $\widetilde{X}$ is normal,
let $\widetilde{D}$~be~the~proper transform on $\widetilde{X}$ of the~divisor $D$,
and let $E_1,\ldots,E_n$ be $\varphi$-exceptional divisors.
Then
$$
K_{\widetilde{X}}+\widetilde{D}+\sum_{i=1}^{n}a(E_i;X,D)E_i\sim_{\mathbb{Q}}\varphi^*\big(K_X+D\big),
$$
where each $a(E_i;X,D)$ is a~rational number known as the~discrepancy of the~pair $(X,D)$ along~$E_i$.
Let $E$ be one of the~$\varphi$-exceptional divisors.
Then
$$
a\big(E;X,D\big)=a(E;X)-\mathrm{ord}_E(D),
$$
where $a(E;X)$  is the~discrepancy of $X$ along $E$.
Let $a=a(E;X)$. If $a(E;X,D)<0$, then
$$
a\Big(E;X,\Big(1+\frac{1}{a}\Big)D\Big)=a(E;X)-\Big(1+\frac{1}{a}\Big)\mathrm{ord}_E(D)
<\mathrm{ord}_E(D)-\Big(1+\frac{1}{a}\Big)\mathrm{ord}_E(D)=-\frac{\mathrm{ord}_E(D)}{a}<-1,
$$
so that the~log pair $(X,(1+\frac{1}{a})D)$ is not log canonical along $\varphi(E)$.
In particular, if $a(E;X,D)<0$ and $\varphi(E)$ is a~smooth point of the~variety $X$,
then the~log pair
$$
\Bigg(X,\frac{\mathrm{dim}(X)}{\mathrm{dim}(X)-1}D\Bigg)
$$
is not log canonical at the~point $\varphi(E)$.
\end{remark}

To prove Proposition \ref{proposition:quadric-maximal-singularities},
we have to present few standard basic facts about the~$\mathfrak{S}_5$-equivariant geometry of the~quadric $Q$.
Observe that $Q$ contains two $\mathfrak{S}_5$-orbits $\Sigma_{10}$ and $\Sigma^\prime_{10}$ of length $10$.

\begin{lemma}
\label{lemma:quadric-G-orbits}
If $\Sigma$ is a~$\mathfrak{S}_5$-orbit in $Q$ with $|\Sigma|<20$,
then $\Sigma$ is one of the~orbits $\Sigma_5$, $\Sigma^\prime_5$, $\Sigma_{10}$, $\Sigma^\prime_{10}$.
\end{lemma}

\begin{proof}
Left to the~reader.
\end{proof}

Let $H=\{x_0+x_1+x_2+x_3+x_4=0\}\subset\mathbb{P}^4$ and $S_2=H\cap Q$.
Then $S_2$ is smooth and $\mathfrak{S}_5$-invariant.
Moreover, the~surface $S_2$ does not contain  $\Sigma_5$, $\Sigma_5^\prime$, $\Sigma_{10}$, $\Sigma^\prime_{10}$.
Let $\mathcal{B}_6$ be the~curve in $Q$ given by
$$
\left\{\aligned
&x_0+x_1+x_2+x_3+x_4=0,\\
&x_0^2+x_1^2+x_2^2+x_3^2+x_4^2=0,\\
&x_0^3+x_1^3+x_2^3+x_3^3+x_4^3=0.\\
\endaligned
\right.
$$
Then $\mathcal{B}_6$ is the~unique smooth curve of genus $4$ that admits an effective action of the~group $\mathfrak{S}_5$,
which is known as the~Bring's curve (see \cite[Remark~5.4.2]{CheltsovShramov}).
Note that $\mathcal{B}_6\subset Q\cap H$.

\begin{lemma}
\label{lemma:quadric-curves}
Let $C$ be a~$\mathfrak{S}_5$-invariant curve in $Q$ such that $\mathrm{deg}(C)\leqslant 6$.
Then $C=\mathcal{B}_6$.
\end{lemma}

\begin{proof}
We may assume that $C$ is $\mathfrak{S}_5$-irreducible, i.e. the~symmetric group $\mathfrak{S}_5$ acts transitively on the~set of its irreducible components.
Then $S_2$ contains $C$, since otherwise $|S_2\cap C|\leqslant S_2\cdot C=12$, which contradicts Lemma~\ref{lemma:quadric-G-orbits}.
Thus, if $C\ne\mathcal{B}_6$, then
$$
|C\cap \mathcal{B}_6|\leqslant C\cdot \mathcal{B}_6=18,
$$
which is impossible by Lemma~\ref{lemma:quadric-G-orbits},
since $S_2$ does no contain $\Sigma_5$, $\Sigma_5^\prime$, $\Sigma_{10}$ and $\Sigma^\prime_{10}$.
\end{proof}

\begin{corollary}
\label{corollary:quadric-log-canonical}
The log pair $(Q,\lambda\mathcal{M}_Q)$ has log canonical singularities.
\end{corollary}

\begin{proof}
Suppose that the~log pair $(Q,\lambda\mathcal{M}_Q)$ is not log canonical. Let us seek for a~contradiction.
If~the log pair $(Q,\lambda\mathcal{M}_Q)$ is log canonical outside of finitely many points,
then it is log canonical outside of a~single point by the~Koll\'ar--Shokurov connectedness,
which must be $\mathfrak{S}_5$-invariant point.
The latter contradicts Lemma~\ref{lemma:quadric-G-orbits}.
Thus, we see that there is a~$\mathfrak{S}_5$-irreducible curve $C$
such that the~log pair $(Q,\lambda\mathcal{M}_Q)$ is not log canonical at general points of its irreducible components.
Then
$$
\big(M_1\cdot M_2\big)_C>\frac{4}{\lambda^2}
$$
by \cite[Theorem 3.1]{corti}, where $M_1$ and $M_2$ are general surfaces in $\mathcal{M}_Q$.
Using this, we get $\mathrm{deg}(C)<\frac{9}{2}$, which is impossible by Lemma \ref{lemma:quadric-curves}.
\end{proof}

Observe that $S_2\cong\mathbb{P}^1\times\mathbb{P}^2$, and the~induced $\mathfrak{S}_5$-action on $S_2$ is faithful.

\begin{lemma}[{cf. \cite[Theorem~7.5]{CheltsovWilson}}]
\label{lemma:quadric-alpha}
One has $\alpha_{\mathfrak{S}_5}(S_2)=\frac{3}{2}$.
\end{lemma}

\begin{proof}
Observe that $\mathrm{Pic}^{\mathfrak{S}_5}(S_2)=\mathbb{Z}[H\vert_{S_2}]$ and $\mathcal{B}_6\in|3H\vert_{S_2}|$.
But $|H\vert_{S_2}|$ and $|2H\vert_{S_2}|$ do not contain any $\mathfrak{S}_5$-invariant curves.
Hence, we have $\alpha_{\mathfrak{S}_5}(S_2)=\frac{3}{2}$ by \cite[Lemma~5.1]{Ch08a} and Lemma~\ref{lemma:quadric-G-orbits}.
\end{proof}

Now we ready to prove Proposition \ref{proposition:quadric-maximal-singularities}.

\begin{proof}[Proof of Proposition~\ref{proposition:quadric-maximal-singularities}]
Suppose $(Q,\lambda\mathcal{M}_Q)$ is not canonical.
Denote by $\Sigma$ its non-canonical locus.
To complete the~proof, we have to show that $\Sigma\subseteq\Sigma_5\cup\Sigma_5^\prime$.

First, let us show that the~set $\Sigma$ consists of finitely many points.
Indeed, suppose that $\Sigma$ contains a~$\mathfrak{S}_5$-irreducible curve $C$.
Then
\begin{equation}
\label{equation:quadric-mult-1}
 \mult_{C}\big(\mathcal{M}_Q\big)>\frac{1}{\lambda},
\end{equation}
which easily implies that $\mathrm{deg}(C)<18$.
Arguing as in the~proof of Lemma~\ref{lemma:quadric-curves}, we see that $C\subseteq S_2$.
Then \eqref{equation:quadric-mult-1} gives $\mathrm{deg}(C)<6$, which is impossible by Lemma \ref{lemma:quadric-curves}.
Hence, we see that $\Sigma$ is finite.

If $\Sigma\cap S_2\ne\varnothing$, then the~log pair $(S_2,\lambda\mathcal{M}_Q|_{S_2})$ is not log canonical by the~inversion of adjunction,
which is impossible by Lemma~\ref{lemma:quadric-alpha}.
Thus, we have $\Sigma\cap S_2=\varnothing$.

Applying Remark~\ref{remark:Ziquan}, we~see~that $(Q,\frac{3\lambda}{2}\mathcal{M}_Q)$ is not log canonical at every point of the~set~$\Sigma$.
Take $\varepsilon\in\mathbb{Q}_{>0}$ such that $\Sigma\subset\mathrm{Nklt}(Q,\frac{3\lambda-\varepsilon}{2}\mathcal{M}_Q)$.
Set $\Omega=\mathrm{Nklt}(Q,\frac{3\lambda-\varepsilon}{2}\mathcal{M}_Q)$.
Then $\Omega$ is~\mbox{$\mathfrak{S}_5$-invariant}.
Moreover, arguing as in the~proof of  Corollary~\ref{corollary:quadric-log-canonical},
we~see that the~locus $\Omega$ does not contain curves, so that $\Omega$ is a~finite set.
Now, applying Nadel vanishing theorem, we get $h^1(Q,\mathcal{J}\otimes \mathcal{O}_Q(2H\vert_{Q}))=0$,
where $\mathcal{J}$ is the~multiplier ideal sheaf of the~log pair $(Q,\frac{3-\varepsilon}{2}\lambda\mathcal{M}_Q)$.
This gives
$$
|\Sigma|\leqslant|\Omega|\leqslant h^0\Big(Q,\mathcal{O}_Q\big(2H\vert_{Q}\big)\Big)=14,
$$
because $\mathrm{Supp}(\mathcal{J})=\Omega$. Now, using Lemma \ref{lemma:quadric-G-orbits}, we see that one of the~following possibilities holds:
\begin{itemize}
\item $\Sigma\subseteq\Sigma_5\cup\Sigma^\prime_5$;
\item $\Omega=\Sigma=\Sigma_{10}$;
\item $\Omega=\Sigma=\Sigma_{10}^\prime$.
\end{itemize}
If $\Sigma\subseteq\Sigma_5\cup\Sigma^\prime_5$, we are done.
Hence, without loss of generality, we may assume that $\Omega=\Sigma=\Sigma_{10}$.
Let us show that this assumption leads to a~contradiction.

Let $\mathcal{D}$ be the~linear subsystem in $|2H|$ that consists of all surfaces in $|2H|$ that pass through~$\Sigma_{10}$.
By counting parameters, we get $\dim(\mathcal{D})\geqslant 4$.
Arguing as in the~proof of Lemma~\ref{lemma:quadric-curves},
we see that the~base locus of the~linear system $\mathcal{D}$ contains no curves.
Using \cite{Pukhlikov97} or \cite[Corollary~3.4]{corti}, we~get
$$
\frac{36}{\lambda^2}=D\cdot M_1\cdot M_2\geqslant\sum_{P\in\Sigma_{10}}\big(M_1\cdot M_2\big)_P>\sum_{P\in\Sigma_{10}}\frac{4}{\lambda^2}=\frac{40}{\lambda^2},
$$
which is absurd. This completes the~proof of Proposition~\ref{proposition:quadric-maximal-singularities}.
\end{proof}

Now, let us present a~few facts about the~threefold $Y$.
Its singular locus consists of five nodes:
\begin{equation*}
\begin{split}
P_1&=(1:0:0:0:0),\\
P_2&=(0:1:0:0:0),\\
P_3&=(0:0:1:0:0),\\
P_4&=(0:0:0:1:0),\\
P_5&=(0:0:0:0:1).\\
\end{split}
\end{equation*}
Note that $(3:3:3:3:-2)\in Y\setminus\mathrm{Sing}(Y)$. Let $\Theta_5$ be the~$\mathfrak{S}_5$-orbit of this point. Then $|\Theta_5|=5$.
For every $1\leqslant i<j\leqslant 5$, we let $\ell_{ij}$ be the~line in $\mathbb{P}^4$ that passes through the~nodes $P_i$ and $P_j$.
Let~$\mathcal{L}_{10}$ be the~union of these lines.
Then $\mathcal{L}_{10}\subset Y$,
and $\mathcal{L}_{10}\cap H$ is a~$\mathfrak{S}_5$-orbit $\Theta_{10}$ of length~$10$.
The cubic $Y$ contains two more $\mathfrak{S}_5$-orbits of~length~$10$, which we denote by $\Theta_{10}^\prime$ and $\Theta_{10}^{\prime\prime}$.

\begin{lemma}
\label{lemma:cubic-G-orbits}
The orbits $\mathrm{Sing}(Y)$, $\Theta_5$, $\Theta_{10}$, $\Theta_{10}^\prime$, $\Theta_{10}^{\prime\prime}$
are all $\mathfrak{S}_5$-orbit in $Y$ of length $<20$.
\end{lemma}

\begin{proof}
Left to the~reader.
\end{proof}

Let $S_3=Y\cap H$. Then $S_3$ is a~smooth cubic surface known as
the Clebsch diagonal cubic surface.
It follows from \cite[Lemma~6.3.12]{CheltsovShramov}
that $\Theta_{10}\subset S_3$,
but $S_3$ does not contain $\mathrm{Sing}(Y)$, $\Theta_5$, $\Theta_{10}^\prime$, $\Theta_{10}^{\prime\prime}$.
Observe also that $S_3$ contains the~curve $\mathcal{B}_6$.

\begin{lemma}
\label{lemma:cubic-curves}
Let $C$ be a~$\mathfrak{S}_5$-invariant curve in $Y$ such that $\mathrm{deg}(C)\leqslant 10$.
Then $C=\mathcal{B}_6$ or $\mathcal{L}_{10}$.
\end{lemma}

\begin{proof}
If $C\subset S_3$, the~assertion follows from \cite[Theorem~6.3.18]{CheltsovShramov}.
Hence, we assume that $C\not\subset S_3$.
Then, arguing as in the~proof of Lemma~\ref{lemma:quadric-curves}, we conclude that  and $C\cdot H=\Theta_{10}$.

We suppose that the~curve $C$ is irreducible. Then $C$ has to be singular at every point $P\in\Theta_{10}$,
because the~stabilizer in $\mathfrak{S}_5$ of the~point $P$ acts faithfully on the~Zariski tangent space $T_{P}(C)$.
Thus, if $C$ is irreducible, then $10=C\cdot H\geqslant 2|\Theta_{10}|$, which is absurd.

We see that $C$ is reducible and $\mathrm{deg}(C)=10$.
Let $C_1$ be an irreducible components of the~curve~$C$,
and let $G$ be the~stabilizer in $\mathfrak{S}_5$ of the~curve $C_1$.
Then one of the~following four cases holds:
\begin{enumerate}
\item $G\cong\mathfrak{A}_5$ and $C$ is a~union of $2$ irreducible curves of degree $5$;
\item $G\cong\mathfrak{S}_4$ and $C$ is a~union of $5$ irreducible conics;
\item $G\cong\mathfrak{A}_4$ the~$C$ is a~union of $10$ lines;
\item $G\cong\mathfrak{S}_3\times\mumu_2$ and $C$ is a~union of $10$ lines.
\end{enumerate}
In the~case (1), $\Theta_{10}$ splits as two $G$-orbits of length $5$, which is not the~case by \cite[Lemma~6.3.12]{CheltsovShramov}.
In the~cases (2) and (3), the~only two-dimensional $G$-invariant linear subspace of $\mathbb{P}^4$
is contained in the~$\mathfrak{S}_5$-invariant hyperplane $H$, so that $C_1$ is contained in $S_3$, which contradicts our assumption.
In the~case (4), one can easily see that $C=\mathcal{L}_{10}$.
\end{proof}

Now, we are ready prove Proposition \ref{proposition:cubic-maximal-singularities}.

\begin{proof}[Proof of Proposition \ref{proposition:cubic-maximal-singularities}]
It is enough to prove that $(Y,\mu\mathcal{M}_Y)$ is canonical away from $\mathrm{Sing}(Y)$.
Suppose that this log pair is not canonical.
Let $\Sigma$ be its non-canonical locus.

First, we claim that $\Sigma$ is a~finite set.
Indeed, suppose that $\Sigma$ contains a~$\mathfrak{S}_5$-irreducible curve.
Then $ \mult_{C}(\mathcal{M})>\frac{1}{\mu}$.
If $C\subset S_3$, this implies that $\mathrm{deg}(C)<6$, which is impossible by Lemma~\ref{lemma:cubic-curves}.
Thus, we see that $C\not\subset S_3$. Then $\mathrm{deg}(C)<12$, so that $H\cdot C<12$.
Using Lemmas~\ref{lemma:cubic-G-orbits} and \ref{lemma:cubic-curves}, we conclude that $C=\mathcal{L}_{10}$.
Let $H^\prime$ be the~hyperplane in $\mathbb{P}^4$ that contains the~nodes $P_1$, $P_2$, $P_3$, $P_4$,
and let $M$ be a~general surface in $\mathcal{M}_Y$.
Then
$$
H^\prime\cdot M=m\big(\ell_{12}+\ell_{13}+\ell_{14}+\ell_{23}+\ell_{24}+\ell_{34}\big)+\Delta
$$
where $a$ is an integer such that $a\geqslant \mult_{C}(\mathcal{M})$,
and $\Delta$ is an effective one-cycle whose support does not contains the~lines
$\ell_{12}$, $\ell_{13}$, $\ell_{14}$, $\ell_{23}$, $\ell_{24}$, $\ell_{34}$.
Therefore, we have
$$
\frac{6}{\mu}=\frac{2}{\mu}H^3=H\cdot H^\prime\cdot M=6m+H\cdot\Delta\geqslant 6m\geqslant 6 \mult_{C}(\mathcal{M})>\frac{6}{\mu},
$$
which is absurd. Thus, we see that $\Sigma$ is a~finite set.

Let $\Sigma_1$ be the~subset in $\Sigma$ that consists of all smooth points of $Y$.
We have to show that~\mbox{$\Sigma_1=\varnothing$}.
If~$\Sigma_1\cap S_3\ne\varnothing$, then the~log pair $(S_3,\mu\mathcal{M}_Y\vert_{S_3})$ is not log canonical,
which implies that $\alpha_{\mathfrak{S}_5}(S_3)<2$. The latter contradicts \cite[Example~1.11]{Ch08a}.
Thus, we have $\Sigma_1\cap S_3=\varnothing$.

Now, using Remark~\ref{remark:Ziquan}, we see that $(Y,\frac{3\mu}{2}\mathcal{M}_Y)$ is not log canonical at every point of the~set $\Sigma_1$.
Moreover, arguing exactly as in the~proof of Corollary~\ref{corollary:quadric-log-canonical} and using Lemma~\ref{lemma:cubic-curves},
we see that each point of the~subset $\Sigma_1$ is an isolated center of non-log canonical singularities of the~pair~$(Y,\frac{3\mu}{2}\mathcal{M}_Y)$.
Now, using Nadel vanishing theorem as we did in the~proof of Proposition~\ref{proposition:quadric-maximal-singularities},
we see that~$|\Sigma_1|\leqslant 5$. Therefore, we have $\Sigma_1=\Theta_5$ by Lemma~\ref{lemma:cubic-G-orbits}.

Let $M_1$ and $M_2$ be general surfaces in $\mathcal{M}_Y$.
Then it follows from \cite{Pukhlikov97} or \cite[Corollary~3.4]{corti} that
$$
\big(M_1\cdot M_2\big)_Q>\frac{4}{\mu^2},
$$
for every point $Q\in \Theta_5$. Let $Q_1$, $Q_2$ and $Q_3$ be three points in $\Theta_5$,
let $\Pi$ be the~plane in $\mathbb{P}^4$ that contains these three points, and let $\mathcal{C}=Y\vert_{\Pi}$.
Then $\mathcal{C}$ is a~smooth irreducible cubic curve.
Write
$$
M_1\cdot M_2=\epsilon\mathcal{C}+\Omega,
$$
where $\epsilon$ is a~non-negative rational number, and $\Omega$ is an effective one-cycle whose support does not contain $\mathcal{C}$.
Let $H^\prime$ be a~general hyperplane section of the~cubic hypersurface $Y$ that contains~$\mathcal{C}$.
Then $H^\prime$ does not contain any irreducible component of the~one-cycle $\Omega$.
Thus, we have
$$
\frac{12}{\mu^2}-3\epsilon=H^\prime\cdot\Omega\geqslant \mult_{Q_1}\big(\Omega\big)+ \mult_{Q_2}\big(\Omega\big)+ \mult_{Q_3}\big(\Omega\big)>3\Big(\frac{4}{\mu^2}-\epsilon\Big)=\frac{12}{\mu^2}-3\epsilon,
$$
which is absurd. This completes the~proof of Proposition~\ref{proposition:cubic-maximal-singularities}.
\end{proof}

This completes the~proof of Theorem~\ref{theorem:quadric-cubic},
which also implies that $Q$ is $\mathfrak{S}_5$-solid \cite{AhmadinezhadOkada,CheltsovDuboulozKishimoto,CheltsovSarikyan}.

\section{Preliminary results}
\label{section:preliminary}

In this section, we prove a few results that will be used towards the proof of Theorem \ref{theorem:main}.

Let $X$ be a~variety with at most Kawamata log terminal singularities that is faithfully acted on by a~finite group $G$.
The following result is a consequence of the technique developed in \cite[\S 3]{Puk-book}.

\begin{lemma}
\label{lemma:hypertangent}
Suppose $X$ is smooth. Let $Z$ be a~$G$-irreducible subvariety of $X$ of codimension~$m$,
let $H$ be an ample divisor on $X$,
and let $D_1,D_2,\ldots,D_m$ be effective divisors on $X$ such that
$$
D_1\sim_{\mathbb{Q}} D_2\sim_{\mathbb{Q}}\cdots\sim_{\mathbb{Q}}D_m\sim_{\mathbb{Q}} H,
$$
and $Z$ is a $G$-irreducible component of the intersection $\cap_{i=1}^m \mathrm{Supp}(D_i)$. Let $Y\subset X$ be an effective cycle of codimension $c\leqslant m$. Then
$$
\frac{\mult_Z(Y)}{\mathrm{deg}(Y)}\leqslant \left(\mathrm{deg} (Z) \cdot \min_S \prod_{i\in S}  \mult_Z (D_i) \right)^{-1}
$$
where the~minimum is taken over all subsets $S\subseteq \{1,\ldots,m\}$ of cardinality $m-c$.
\end{lemma}

\begin{proof}
We may assume  that $Y$ is irreducible and $Z\subseteq Y$. We construct a~sequence of irreducible subvarieties $Y_c,\ldots,Y_m$ and a~permutation $D'_1,\ldots,D'_m$ of $D_1,\ldots,D_m$ such that
\begin{itemize}
\item $Y_c=Y$;
\item $\mathrm{codim}_X(Y_i)=i$;
\item $Y_i\not\subset \mathrm{Supp}(D'_{i+1})$;
\item $Y_{i+1}$ is a~component of $Y_i\cdot D'_{i+1}$ that contains $Z$;
\item for all $c\leqslant i\leqslant m-1$ one has
 $$
\frac{\mult_Z(Y_{i+1})}{\mathrm{deg}(Y_{i+1})}\geqslant \mult_Z (D'_{i+1}) \cdot \frac{\mult_Z(Y_i)}{\mathrm{deg}(Y_i)}.
 $$
\end{itemize}
Once this is done, the~lemma follows immediately from the~trivial equality $Y_m=Z$.

Suppose that $Y_c, \ldots, Y_i$ and $D'_{c+1},\ldots,D'_i$ have been constructed for some $i<m$. Then
$$
Y_i\subseteq \bigcap_{j=c+1}^{i} \mathrm{Supp}(D_j).
$$
Since $\cap_{i=1}^m \mathrm{Supp}(D_i)$ has codimension $m$ in a~neighbourhood of $Z$ by assumption and
$$
\mathrm{codim}_X(Y_i)=i<m,
$$
then there exists some $D_j$, which is necessary different from $D'_{c+1},\ldots,D'_i$,
which gives $Y_i\not\subseteq D_j$. We may then take $D'_{i+1}=D_j$ and $Y_{i+1}$ an irreducible component of $(Y_i\cdot D_j)$ such that
$$
\frac{\mult_Z(Y_{i+1})}{\mathrm{deg}(Y_{i+1})}\geqslant \frac{\mult_Z(Y_i\cdot D_j)}{\mathrm{deg}(Y_i\cdot D_j)}\geqslant \mult_Z D_j \cdot \frac{\mult_Z(Y_i)}{\mathrm{deg}(Y_i)}.
$$
By induction this finishes the~construction.
\end{proof}

Now, let $D$ be either an effective $\mathbb{Q}$-divisor on $X$ (a boundary),
or a~movable (mobile) boundary:
$$
D=\sum_{i=1}^{r}a_i\mathcal{M}_i,
$$
where each $a_i\in\mathbb{Q}_{\geqslant 0}$, and each $\mathcal{M}_i$ is a~linear system on $X$ that does not have fixed components.
Suppose, in addition, that $D$ is $G$-invariant.

\begin{lemma}
\label{lemma:tie-breaking}
Suppose that $(X,D)$ is not log canonical, and $D$ is ample.
Then there exists positive rational number $\epsilon<1$ such that the~following assertions hold:
\begin{itemize}
\item if $D$ is a~$\mathbb{Q}$-divisor, there exists a~$G$-invariant effective $\mathbb{Q}$-divisor $D^\prime\sim_{\mathbb{Q}} (1-\epsilon)D$
such that the~log pair $(X,D^\prime)$ has log canonical singularities, and $\mathrm{Nklt}(X,D^\prime)$ is a~non-empty disjoint union of minimal log canonical centers
of the~log pair $(X,D^\prime)$,
\item if $D$ is a~mobile boundary, there exists a~$G$-invariant mobile boundary $D^\prime\sim_{\mathbb{Q}} (1-\epsilon)D$
such that the~log pair $(X,D^\prime)$ has log canonical singularities, and $\mathrm{Nklt}(X,D^\prime)$ is a~non-empty disjoint union of minimal log canonical centers
of the~log pair $(X,D^\prime)$.
\end{itemize}
Furthermore, irreducible components of $\mathrm{Nklt}(X,D^\prime)$ are normal,
and $G$ transitively permutes them.
\end{lemma}

\begin{proof}
This is an~equivariant version of~\emph{the tie breaking}. See \cite[Lemma~2.4.10]{CheltsovShramov} or \cite{Kawamata97,Kawamata98}.
\end{proof}

\begin{lemma}
\label{lemma:nNklt-base-locus}
Let $H$ be  a~very ample divisor in $\mathrm{Pic}(X)$, and let $L$ be a~divisor in $\mathrm{Pic}(X)$
such that the~divisor $L-(K_X+D+\dim(X)H)$ is ample.
Then $|L|$ contains a~non-empty $G$-invariant linear subsystem $\mathcal{L}$ such that $\mathrm{Nklt}(X,D)=\mathrm{Bs}(\mathcal{L})$.
\end{lemma}

\begin{proof}
Let $\cJ=\cJ(X,D)$ be the~multiplier ideal. Then the~support of $\cO_X/\cJ$ is exactly $\mathrm{Nklt}(X,D)$.
By \cite[Proposition 9.4.26]{La04}, $\cJ\otimes \cO_X(L)$ is generated by global sections. The $G$-invariant linear system $\cL=|\cJ\otimes \cO_X(L)|$ then satisfies the~statement of the~lemma.
\end{proof}

Now, we fix $d,n \in \mathbb{Z}_{>0}$.
Let $\mathbb{W}$ be the~subgroup in $\mathrm{GL}_{n+1}(\mathbb{C})$ consisting of all permutation matrices, let $\mathbb{T}$ be the~subgroup in $\mathrm{GL}_{n+1}(\mathbb{C})$ consisting of diagonal matrices whose (non-zero) entries are $d$-th roots~of~unity,
and let $\mathbb{G}$ be the~subgroup in $\mathrm{GL}_{n+1}(\mathbb{C})$ generated by $\mathbb{T}$ and $\mathbb{W}$.
Then $\mathbb{W}\cong\mathfrak{S}_{n+1}$, $\mathbb{T}\cong\mumu_d^{n+1}$, and
$$
\mathbb{G}\cong\mathbb{T}\rtimes\mathbb{W}\cong\mumu_d^{n+1}\rtimes\mathfrak{S}_{n+1}.
$$
Let~$W$, $T$, $G$ be the~images in $\mathrm{PGL}_{n+1}(\mathbb{C})$ via the~quotient map of the~groups $\mathbb{W}$, $\mathbb{T}$, $\mathbb{G}$, respectively.
Then $W\cong\mathfrak{S}_{n+1}$, $T\cong\mumu_d^n$, and $G\cong T\rtimes W$.
Note~that $G$ leaves invariant the~Fermat hypersurface
$$
X_d\vcentcolon=\Big\{\sum_{i=0}^{d}x_i^d=0\Big\}\subset\mathbb{P}^{n},
$$
where $x_0,\ldots,x_n$ are homogeneous coordinates on $\mathbb{P}^n$.
If $n\geqslant 2$ and $d\geqslant 3$, then $G=\mathrm{Aut}(\mathbb{P}^n,X_d)$.

The examples for Theorem \ref{theorem:main} are complete intersections in $\mathbb{P}^{n}$ of some Fermat hypersurfaces.
The main result of this section is the~following proposition. We will use it in the~next section.

\begin{proposition}
\label{proposition:technical}
Let $\mathcal{M}$ be a~$W$-invariant linear subsystem in $|\mathcal{O}_{\mathbb{P}^n}(m)|$,
let $Z$ be an~irreducible component of the~intersection $\mathrm{Bs}(\mathcal{M})$,
and let $\mathscr{Z}$ be the~$W$-irreducible subvariety in $\mathbb{P}^n$, whose irreducible component is $Z$.
Then at least one of the~following two cases holds:
\begin{enumerate}
\item[(1)] a general point in $Z$ has at most $d$ different coordinates, and $\dim(Z)\leqslant m-1$;
\item[(2)] the~subvariety $Z$ is an irreducible component of a~set-theoretic intersection of $W$-invariant hypersurfaces of degree at most $m$, and $\dim(Z)\geqslant n-m$.
\end{enumerate}
Moreover, in case $\mathrm{(1)}$, if $m\leqslant n$ and $n\geqslant 4$, then either $\mathscr{Z}$ has at least $n+1$ irreducible components,
or $\mathscr{Z}=Z=(1:1:\ldots:1)$.
\end{proposition}

In particular, the~base locus of a~$W$-invariant linear subsystem in $|\mathcal{O}_{\mathbb{P}^n}(m)|$ either has dimension at most $m-1$, or has codimension at most $m$. This can be illustrated by the~following example.

\begin{example}
\label{example:technical}
In the~assumptions and notations of Proposition~\ref{proposition:technical}, suppose  $m=1$ and $Y=\mathbb{P}^n$.
Then either $\mathcal{M}=|\mathcal{O}_{\mathbb{P}^n}(1)|$, so it is base point free,
or one of the~following two cases holds:
\begin{enumerate}
\item[(1)] $\mathcal{M}$ is the~linear system of hyperplanes containing the~$W$-invariant point $(1:1:\ldots:1)$;
\item[(2)] $\mathcal{M}$ is the~$W$-invariant hyperplane $X_1=\{x_0+\ldots+x_n=0\}\subset\mathbb{P}^n$.
\end{enumerate}
In case (1), we have $\mathscr{Z}=Z=(1:1:\ldots:1)$. In case (2), we have $\mathscr{Z}=Z=X_1$.
\end{example}

To prove Proposition \ref{proposition:technical}, we need to prove a~few auxiliary results.

\begin{lemma}
\label{lemma:partition-number}
Fix $s\in\{1,\ldots,n\}$, and take positive integers $a_1,\ldots,a_s$  such that $n=a_1+\ldots+a_s$.
Let $N$ be the~number of unordered partitions of the~set $\{1,\ldots,n\}$ into subsets of $a_1,\ldots,a_s$ elements, respectively.
Then $N\geqslant n$ unless $s=1$, $s=n$, or $n=4$, $s=2$, $a_1=a_2=2$.
\end{lemma}

\begin{proof}
We may assume $a_1\leqslant a_2\leqslant \ldots \leqslant a_k$.
If $a_1=\ldots=a_i<a_{i+1}$ for some $i\in\{1,\ldots,k-1\}$,~then
$$
N\geqslant\binom{n}{ia_1}\geqslant n.
$$
Hence, we may assume that $a_1=\ldots=a_s=r$ for some $r\in\{2,\ldots,n-1\}$. Then $s=\frac{n}{r}\geqslant 2$ and
$$
N=\frac{n!}{(r!)^s \cdot s!}\geqslant \frac{n(n-1)\cdot \ldots \cdot (n-r+1)}{r!\cdot s}.
$$
Thus, since $n\geqslant 2s$, we get
$$
N\geqslant (n-1)\cdot \frac{n}{2s}\cdot \prod_{j=0}^{r-3} \frac{n-2-j}{r-j}\geqslant n-1,
$$
with equality only if $n=2s$ and $n-2=r$, i.e. when $r=s=2$. Since $N$ is a~positive integer, the~assertion follows.
\end{proof}

For the~second result, we need the~following two conventions.
A \emph{color set} is a~finite multiset, where elements (i.e. colors) may appear with multiplicities.
If $K=(V,E)$ is a~graph and $\mathscr{C}$ is a~color set,
then a~\emph{coloring} of the~graph $K$ by $\mathscr{C}$ is a~map $\phi\colon V\to\mathscr{C}$ such that
\begin{itemize}
\item every color is used at most once, i.e. we have $|\phi^{-1}(c)|\leqslant 1$ for every $c\in \mathscr{C}$,
\item and every pair of adjacent vertices have different color, i.e. we have $\phi(u)\neq \phi(v)$ as integers whenever $(u,v)\in E$.
\end{itemize}

\begin{lemma}
\label{lemma:coloring}
Let $K=(V,E)$ be a~graph such that $K$ contains at least $s\geqslant 1$ connected components,
and let $\mathscr{C}$ be a~color set of size at least $|V|$ such that $\mathscr{C}$ has at least $|V|-s+1$ different colors.
Then there exists a~coloring of the~graph $K$ by $\mathscr{C}$.
\end{lemma}

\begin{proof}
We use induction on $|V|-s\geqslant 0$.
The result is clear when $|V|-s=0$, since in this case there are no edges in $K$.
Suppose now that the~result has been proved for smaller values of $|V|-s$.
We can assume that every connected component of $K$ contains at least $2$ vertices, since we can assign any color to isolated points. In particular, the~number $|V|-s$ drops if we remove connected components from $V$. It is also clear that we may assume $s\geqslant 2$ and at least one of the~colors has multiplicity $\geqslant 2$ (otherwise there are already $|V|$ different colors).

Now, we let $K_1=(V_1,E_1)$ be a~connected component of the~graph $K$, and we set $r=|V_1|\geqslant 2$.
Let $K^\prime=(V^\prime,E^\prime)$ be the~subgraph of the~graph $K$ that is obtained by removing the~component~$K_1$.
We may choose a~subset $\mathscr{C}_1\subseteq\mathscr{C}$ that consists of $r$ distinct colors (each with multiplicity 1) such that
the~complement $\mathscr{C}\backslash\mathscr{C}_1$ has at least $|V|-s+2-r$ different color (here we use the~assumption that at least one color in $\mathscr{C}$ has multiplicity $\geqslant 2$).
Note that we can color the~graph $K_1$ by $\mathscr{C}_1$.
By~induction hypothesis, we can also color $K^\prime$ by $\mathscr{C}\backslash\mathscr{C}_1$, since $|V|-s+2-r=|V'|-(s-1)+1$.       This gives us a~coloring of $K$ by $\mathscr{C}$.
\end{proof}

Let us identify $H^0(\mathbb{P}^n,\mathcal{O}_{\mathbb{P}^n}(m))$ with
the subspace in $\mathbb{C}[x_0,\ldots,x_n]$ consisting of all homogeneous polynomials of degree $m$.
For $f\in H^0(\mathbb{P}^n,\mathcal{O}_{\mathbb{P}^n}(m))$ and a~(possibly reducible) subvariety $Y\subset\mathbb{P}^n$,
we~define $f\vert_{Y}$ to be the~image of the~polynomial $f$ in $H^0(Y,\mathcal{O}_{\mathbb{P}^n}(m)\vert_{Y})$ via the~restriction morphism.
For any $f\in H^0(\mathbb{P}^n,\mathcal{O}_{\mathbb{P}^n}(m))$,
we denote by $\mathcal{M}_f$ the~linear subsystem in $|\mathcal{O}_{\mathbb{P}^n}(m)|$
that is given by the~supspace in $H^0(\mathbb{P}^n,\mathcal{O}_{\mathbb{P}^n}(m))$ spanned by $\tau^*(f)$ for all $\tau\in\mathbb{W}$.
Finally, we fix $V=\{0,1,\ldots,n\}$.
For every graph $K=(V,E)$,
let $c(K)$ be the~number of its connected components, and let
$$
f_K=\pm \prod_{(i,j)\in E} (x_i-x_j)\in H^0\big(\mathbb{P}^n,\mathcal{O}_{\mathbb{P}^n}(|E|)\big).
$$

\begin{lemma}
\label{lemma:claim}
Let $Y$ be an intersection in $\mathbb{P}^n$ of some $W$-invariant hypersurfaces,
and fix $\ell\in\mathbb{Z}_{>0}$.
Take~some $g\in H^0(\mathbb{P}^n,\mathcal{O}_{\mathbb{P}^n}(\ell))$ such that $g\vert_{Y}$ is not $\mathbb{W}$-invariant,
and let~$K=(V,E)$ be a~graph.
Then there exists a~graph $K^\prime=(V,E^\prime)$ containing $K$ as a~subgraph
and $g^\prime\in H^0(\mathbb{P}^n,\mathcal{O}_{\mathbb{P}^n}(\ell-1))$
such that $c(K^\prime)\geqslant c(K)-1$, $g^\prime|_Y\neq 0$ and
$$
\mathrm{Bs}\big(\mathcal{M}_h\big)\subseteq\mathrm{Bs}\big(\mathcal{M}_{h^\prime}\big)
$$
for $h=f_K g$ and $h^\prime=f_{K^\prime} g^\prime$.
\end{lemma}

\begin{proof}
Since $g|_Y$ is not $\mathbb{W}$-invariant, there is a~transposition $\tau=(ij)\in\mathbb{W}$ such that $\tau^*(g)|_Y\neq g|_Y$.
Then $\tau^*(g)-g=(x_i-x_j) g^\prime$
for some $g^\prime\in H^0(\mathbb{P}^n,\mathcal{O}_{\mathbb{P}^n}(\ell-1))$ such that $g^\prime|_Y\neq 0$,
since otherwise we would have $\tau^*(g)|_Y=g\vert_{Y}$.

Let $\tau(K)$ be the~graph obtained from $K$ by switching the~labeling of the~vertices $i$ and $j$ without changing any edges,
and let $K^\prime$ be the~graph obtained by adding the~edge $(ij)$ to $K\cup \tau(K)$ (take the~union of edges).
Then $c(K^\prime)\geqslant c(K)-1$.

Let $h=f_Kg$ and $h^\prime=f_{K^\prime}g^\prime$.
Then $\tau^*(h)=f_{\tau(K)}\tau^*(g)$,
and $\mathrm{Bs}(\mathcal{M}_h)\subseteq \mathrm{Bs}(\mathcal{M}_{h^\prime})$,
because $h^\prime$ has the~same factors (ignoring multiplicities) as
$$
f_{K\cup \tau(K)}\cdot (x_i-x_j)g^\prime= f_{K\cup \tau(K)}\big(\tau^*(g)-g\big)=f_{K-\tau(K)}\tau^*(h)-f_{\tau(K)-K} h,
$$
where $K-\tau(K)$ is the~graph obtained by removing from $K$ the~edges of $\tau(K)$.
\end{proof}

\begin{corollary}
\label{corollary:claim}
Let $Y$ be an intersection in $\mathbb{P}^n$ of $W$-invariant hypersurfaces,
let $f$ be a~polynomial in $H^0(\bP^n,\mathcal{O}_{\mathbb{P}^n}(m))$ such that $f|_Y\neq 0$.
Then there are $r\in\{0,1,\ldots,m\}$, a~graph \mbox{$K=(V,E)$},
and a~$\mathbb{W}$-invariant polynomial $g_0\in H^0(\mathbb{P}^n,\mathcal{O}_{\mathbb{P}^n}(m-r))$ such that $g_0|_Y\neq 0$,
$c(K)\geqslant n+1-r$, and
\begin{equation}
\label{equation:base-loci}
\mathrm{Bs}\big(\mathcal{M}_f\big)\cap Y\subseteq\mathrm{Bs}\big(\mathcal{M}_{g}\big)
\end{equation}
for $g=f_K g_0$.
\end{corollary}

\begin{proof}
Let us apply Lemma~\ref{lemma:claim} repeatedly  starting with the~graph $(V,\varnothing)$ and $g=f$.
This process must stop after at most $m$ steps.
Therefore, we obtain a~graph $K(V,E)$ and a~polynomial $g=f_Kh$ such that $\mathrm{deg}(h)=m-r$ for $r\leqslant m$,
$c(K)\geqslant n+1-r$, the~restriction $h\vert_{Y}$ is $\mathbb{W}$-invariant, and
$$
\mathrm{Bs}\big(\mathcal{M}_f\big)\subseteq\mathrm{Bs}\big(\mathcal{M}_{g}\big).
$$
Then we can replace $h$ by a~$W$-invariant polynomial $g_0$ of the~same degree such that $g_0|_Y=h|_Y$.
\end{proof}

\begin{proof}[Proof of Proposition \ref{proposition:technical}]
The assertions on $\mathrm{dim}(Z)$ and the~assertion on the~number of irreducible components of the~subvariety $\mathscr{Z}$
follow from Lemma \ref{lemma:partition-number}. Thus, we have to prove that
\begin{enumerate}
\item[(1)] either a general point in $Z$ has at most $m$ different coordinates,
\item[(2)] or the~subvariety $Z$ is an irreducible component of a~set-theoretic intersection of $W$-invariant hypersurfaces of degree at most $m$.
\end{enumerate}

Let $D_1,\dots,D_k$ be $W$-invariant hypersurfaces of degree at most $m$ that contain $Z$,
and let $Y$ be the~set theoretic intersection $D_1\cap \dots\cap D_k$ (if there exist no such hypersurfaces, we set $Y=X$).
We~may assume $Z\subsetneqq Y$ (otherwise (2) clearly holds).
Hence, there is $f\in \cM$ such that $f|_Y\neq 0$. Note that this gives $Z\subseteq \Bs(\cM_f)\cap Y$.

By Corollary~\ref{corollary:claim}, we find a~graph $K=(V,E)$ with  $c(K)\geqslant n+1-r$ and a~$\mathbb{W}$-invariant polynomial $g_0\in H^0(\mathbb{P}^n,\mathcal{O}_{\mathbb{P}^n}(m-r))$
such that $g_0|_Y\neq 0$ and \eqref{equation:base-loci} holds, where $r\in\{0,1,\ldots,m\}$. By the~construction of $Y$, we see that $g_0|_Z\neq 0$, thus \eqref{equation:base-loci} gives
$$
Z\subseteq \mathrm{Bs}\big(\mathcal{M}_{f_K}\big).
$$
Pick a~general point $z\in Z$ with coordinates $[z_0:\ldots:z_n]$ and consider the~color set $\cC=\{z_0,\dots,z_n\}$.
If (1) does not hold, then there are at least $m\geqslant r$ different colors in $\cC$.
By Lemma~\ref{lemma:coloring}, we may color the~graph $K$ by $\cC$.
After unwinding the~definitions, this implies that there is $\sigma\in\mathbb{W}$ such that $\sigma^*(f_K)$ does not vanish on $Z$.
But this is a~contradiction as $\sigma^*(f_K)\in \cM_{f_K}$. So, we conclude that (1) holds in this case and this completes the~proof of the~proposition.
\end{proof}

Let us apply Proposition \ref{proposition:technical}.
Recall that $X_d$ is the~Fermat hypersurface in $\mathbb{P}^{n}$ of degree $d$.

\begin{proposition}
\label{proposition:Fermat-hypersurfaces}
If $d\leqslant n \leqslant 3d-1$, then $\alpha_G(X)\geqslant 1$.
If $n\geqslant 3d$, then $\alpha_G(X)=\frac{2d}{n+1-d}$.
\end{proposition}

\begin{proof}
We suppose that $n\geqslant d$. Let $H$ be a~hyperplane section of $X_d$, let $r=\min\{2d,n+1-d\}$,
and let $D$ be a~$G$-invariant effective divisor on $X_d$ such that $D\sim_{\mathbb{Q}} rH$.
We have $\alpha_G(X)\leqslant \frac{2d}{n+1-d}$, where the~right hand side is computed by the~$G$-invariant Fermat hypersurface $X_{2d}$.
Hence, both statements of the~proposition would follow once we prove that the~log pair $(X,D)$ is log canonical.
Suppose that $(X,D)$ is not log canonical. Let us seek for a~contradiction.

Let $\lambda=\mathrm{lct}(X,D)$ and $Z=\mathrm{Nklt}(X,\lambda D)$.
Then $(X,\lambda D)$ is log canonical, $\lambda<1$ and $Z\ne\varnothing$.
Applying Lemma \ref{lemma:tie-breaking}, we may assume that $Z$ is a~disjoint union of irreducible normal subvarieties.
But, on the~other hand, since $-(K_X+\lambda D)$ is ample, applying Koll\'ar--Shokurov's connectedness,
we~conclude that $Z$ is an~irreducible subvariety.
By Lemma \ref{lemma:nNklt-base-locus}, there exists a~$G$-invariant linear subsystem
$\mathcal{L}\subset|(3d-2)H|$ such that $Z=\mathrm{Bs}(\mathcal{L})$.

Let $V$ be the~vector subspace in $H^0(X,\mathcal{O}_X((3d-2)H))$ that corresponds to the~linear system~$\mathcal{L}$.
Then $V$ is a~$\mathbb{G}$-subrepresentation in $H^0(X,\mathcal{O}_X((3d-2)H))$.
As $\mathbb{T}$-representation, we have
$$
V=\bigoplus_{\chi} V_{\chi},
$$
where the~summand runs over all characters $\chi$ of the~group $\mathbb{T}$.
For each $\chi$, we have $V_{\chi}=\mathbf{x}_{\chi}\cdot W_{\chi}$,
where $\mathbf{x}_{\chi}$ is a~monomial of degree at most $d-1$ in each homogeneous coordinate $x_0,x_1,\ldots,x_{n}$,
while $\bT$ acts trivially on $W_{\chi}$.
Each $W_{\chi}$ is the~image in $H^0(X,\mathcal{O}_X(mdH))$ of a subspace of
$$
\mathrm{Sym}^m(U)\subseteq H^0(\mathbb{P}^n,\mathcal{O}_{\mathbb{P}^n}(md)),
$$
where $U=\mathrm{span}(x_0^d,\ldots,x_n^d)\subseteq H^0(\mathbb{P}^n,\mathcal{O}_{\mathbb{P}^n}(d))$ and $m\leqslant 2$,
because $md=\mathrm{deg}(W_\chi)\leqslant 3d-2$.

Since the~action of the~group $\mathbb{G}$ on the~vector space $H^0(\mathbb{P}^n,\mathcal{O}_{\mathbb{P}^n}(1))$ is irreducible,
we see that the~subvariety $Z$ is not contained in a~hyperplane, so $Z$ is not contained in $\{\mathbf{x}_{\chi}=0\}$ for any~$\chi$.
Then $Z$ is a~set-theoretic intersection of zeroes of all polynomials in all $W_{\chi}$. Since $Z$ is invariant under the~$G$-action, we see that $\sigma^*(f)$ vanishes on $Z$ for any $f\in W_{\chi}$ and any $\sigma\in\mathbb{W}\cong\mathfrak{S}_{n+1}$.

Now, let us consider a~morphism $\upsilon\colon \mathbb{P}^n \to \mathbb{P}^n$ defined as
$$
\upsilon(x_0: \ldots : x_n)=\big(x_0^d:\ldots:x_n^d\big).
$$
Then the~induced action of $G$ on $\mathrm{Im}(f)=\mathbb{P}^n$ is isomorphic (as an~action) to the~permutational action of the~group $W\cong\mathfrak{S}_{n+1}$ on $\mathbb{P}^n$.
Further, we observe that
$$
(1:\ldots:1)\not\in \upsilon(X_d).
$$
Moreover, since~$Z$ is connected and $W$-invariant, so is $\upsilon(Z)$.
Thus, from the~previous discussion, we conclude that $Z$ is the~base locus of some $W$-invariant linear system of degree at most $2d$,
generated by all the~polynomials in $W_{\chi}$.
Then $\upsilon(Z)$ is the~base locus of some $W$-invariant linear system of degree at most $2$.
Applying Proposition~\ref{proposition:technical} to $\upsilon(Z)$,
we see that $Z$ is an~irreducible component of a~set-theoretic intersection of $G$-invariant hypersurfaces of degree at most $2d$.
Then
$$
Z=X_d\cap X_{2d}.
$$
On the~other hand, since the~log pair $(X_d,D)$ is not log canonical along $Z$,
we have $\mult_{Z}(D)>1$. This contradicts to $Z\sim_{\mathbb{Q}} 2dH$ and $r\leqslant 2d$.
\end{proof}

Similarly, we prove the~following result.

\begin{proposition}
\label{proposition:Fermat-cpi}
Let $X=X_d\cap X_{2d}\cap\ldots\cap X_{rd}$ for $r\geqslant 1$,
let $H$ be a~hyperplane section of~$X$, and let $D$ be a~$G$-invariant effective $\mathbb{Q}$-divisor on $X$ such that $D\sim_{\mathbb{Q}} qH$
for a~positive rational number $q<(r+1)d$.
Suppose, in addition, that $\mathrm{dim}(X)\geqslant 1$, $n\geqslant 4$, and $dH-(K_X+D)$~is~nef.
Then the~log pair $(X,D)$ is log canonical.
\end{proposition}

\begin{proof}
Replacing $q$ by $\lceil q \rceil$, and $D$ by $\frac{\lceil q \rceil}{q}D$,
we may assume that the~number $q$ is actually integer.
Suppose that $(X,D)$ is not log canonical. Let us seek for a~contradiction.

Let $\lambda=\mathrm{lct}(X,D)$ and $Z=\mathrm{Nklt}(X,\lambda D)$.
Then $(X,\lambda D)$ is log canonical, $\lambda<1$ and $Z\ne\varnothing$.
Applying Lemma \ref{lemma:tie-breaking}, we may assume that $Z$ is a~disjoint union of irreducible normal subvarieties,
and $Z$ is $G$-irreducible.
Moreover, using Lemma \ref{lemma:nNklt-base-locus}, we see that $Z=\mathrm{Bs}(\mathcal{L})$
for some $G$-invariant linear subsystem
$\mathcal{L}\subset|aH|$, where
$$
a=\frac{r(r+1)}{2}d+q-(r-1)
$$
that satisfies $K_X+D+(n-r)H\sim_{\mathbb{Q}} aH$.

Now, let $V$ be the~vector subspace in $H^0(X,\mathcal{O}_X(aH))$ that corresponds to the~linear system~$\mathcal{L}$.
Then $V$ is a~$\mathbb{G}$-subrepresentation in $H^0(X,\mathcal{O}_X(aH))$.
As before, we have
$$
V=\bigoplus_{\chi} V_{\chi},
$$
where the~summand runs over all characters $\chi$ of the~group $\mathbb{T}$.
For each $\chi$, we have $V_{\chi}=\mathbf{x}_{\chi}\cdot W_{\chi}$,
where $\mathbf{x}_{\chi}$ is a~monomial of degree at most $d-1$ in each homogeneous coordinate $x_0,x_1,\ldots,x_{n}$,
and each $W_{\chi}$ is the~image in $H^0(X,\mathcal{O}_X(\ell dH))$ of a subspace of
$$
\mathrm{Sym}^{\ell}(U)\subseteq H^0(\mathbb{P}^n,\mathcal{O}_{\mathbb{P}^n}(\ell d)),
$$
where $U=\mathrm{span}(x_0^d,\ldots,x_n^d)\subseteq H^0(\mathbb{P}^n,\mathcal{O}_{\mathbb{P}^n}(d))$ and $\ell\leqslant \lfloor \frac{a}{d}\rfloor$.

Now, let $Z_1,\ldots,Z_s$ be the~$T$-irreducible components of the~locus $Z$.
Then we claim that $s\leqslant n$. Indeed, using Nadel vanishing theorem and the~nefness of the~divisor $dH-(K_X+D)$,
we get
$$
H^1\Big(X,\mathcal{J}\big(X,\lambda D\big)\otimes \mathcal{O}_X\big(dH\big)\Big)=0,
$$
where $\mathcal{J}(X,\lambda D)$ is the~multiplier ideal sheaf of the~log pair $(X,\lambda D)$.
Now, let $\Upsilon$ be the~subscheme defined by the~multiplier ideal sheaf $\mathcal{J}(X,\lambda D)$ of the log  pair $(X,\lambda D)$,
and let $\Upsilon_i$ be its irreducible component supported on $Z_i$ for $i\in\{1,\ldots,s\}$.
Then the~natural restriction
$$
H^0\Big(X,\mathcal{O}_X\big(dH\big)\Big)\rightarrow H^0\Big(\Upsilon,\mathcal{O}_{\Upsilon}\big(dH\vert_{\Upsilon}\big)\Big)
$$
is surjective. Taking the~$\mathbb{T}$-invariant parts, we see that
\begin{multline*}
s\leqslant\sum_{i=1}^{s}\mathrm{dim}\Bigg(H^0\Big(\Upsilon_i,\mathcal{O}_{\Upsilon_i}\big(dH\vert_{\Upsilon_i}\big)\Big)^{\mathbb{T}}\Bigg)
\leqslant \mathrm{dim}\Bigg(H^0\Big(\Upsilon,\mathcal{O}_{\Upsilon}\big(dH\vert_{\Upsilon}\big)\Big)^{\mathbb{T}}\Bigg)\leqslant \\ \leqslant\mathrm{dim}\Bigg(H^0\Big(X,\mathcal{O}_X\big(dH\big)\Big)^{\mathbb{T}}\Bigg)=\mathrm{dim}(U)-1 =n.\quad\quad\quad\quad\quad
\end{multline*}
Here, the~first inequality holds because $H^0(\Upsilon_i,\mathcal{O}_{\Upsilon_i}(dH\vert_{\Upsilon_i}))^{\mathbb{T}}$ contains $U\vert_{\Upsilon_i}\neq 0$.

We claim that no $T$-irreducible components of  $Z$ are contained in coordinate hyperplanes.
\mbox{Indeed}, otherwise, such component would be contained in the~(unique) minimal $T$-invariant linear subspace in $\mathbb{P}^n$,
which would imply that $Z$ has at least $n+1$ $T$-irreducible components.

Let $m=\lfloor \frac{a}{d}\rfloor$.~Arguing as in the proof of Proposition~\ref{proposition:Fermat-hypersurfaces},
we see that $Z$ is the~base locus of the~$W$-invariant subsystem of $|\cO_{\bP^n}(md)|$ generated by $\mathrm{Bs}(W_\chi)$
and hypersurfaces containing~$X$.
Now, using Proposition \ref{proposition:technical} and the~same morphism $\upsilon\colon \mathbb{P}^{n}\to \mathbb{P}^n$ as in Proposition \ref{proposition:Fermat-hypersurfaces},
we~conclude (as in the~proof of Proposition \ref{proposition:Fermat-hypersurfaces})
that $Z$ is a~$G$-irreducible component of the~set-theoretic intersection of some $G$-invariant hypersurfaces of degree at most $md$,
because the other possibility in Proposition \ref{proposition:technical} is excluded, since $\upsilon(Z)$ has at most $n$ irreducible components and
$$
\big\{x_0^d=x_1^d=\ldots=x_n^d\big\}\not\subset X.
$$
In particular, we see that the~pair $(X,D)$ is not log canonical along $Y=X_d\cap X_{2d}\cap \ldots \cap X_{md}$,
and hence $\mult_Y (D)>1$. But as $n>m$ under our assumption (we leave to the~reader to verify this),
we see that $Y$ is irreducible. Now, applying Lemma \ref{lemma:hypertangent} to $D_k=\frac{1}{kd}(X_{kd}\cdot X)$ for $k=r+1,\ldots,m$,
we get
$$
\mult_Y (D) \leqslant \frac{\mathrm{deg} (D)}{\mathrm{deg} (Y)}\prod_{i=k+2}^{m} kd \leqslant \frac{\mathrm{deg} (H)}{\mathrm{deg} (Y)}\prod_{k=r+1}^{m} kd = 1,
$$
which is a contradiction.
\end{proof}

\section{The proof of Theorem~\ref{theorem:main}}
\label{section:the-proof}

Let us use all assumptions and notations of Section~\ref{section:preliminary}.
Let $X$ be the~complete intersection in the~projective space $\mathbb{P}^n$ of the~Fermat hypersurfaces $X_{2d},X_{3d},\ldots,X_{rd}$ for some integer $r\geqslant 2$,
and let $H$ be a~hyperplane section of the~variety $X$.
Suppose that
$$
-K_X\sim qH
$$
for some $q\leqslant \frac{(r+1)d}{2}$. Then  $\alpha_G(X)\leqslant \frac{d}{q}$, since $-K_X\sim_\bQ \frac{q}{d}X_d|_X$. So, we can make $\alpha_G(X)$ arbitrarily small by choosing $q=\lfloor \frac{1}{2}(r+1)d\rfloor$ and letting $r\gg 0$.
Therefore, to prove Theorem \ref{theorem:main}, it remains to show that $X$ is $G$-birationally super-rigid.

In order to prove this, we use a~similar strategy as in Proposition~\ref{proposition:Fermat-cpi}. Let $\M$ be a~$G$-invariant mobile linear system, and let $\lambda$ be a~positive rational number such that
$$
K_X+\lambda \M \sim_{\Q} 0.
$$
As in the~proof of the~Theorem \ref{theorem:alpha} we need to show that $(X,\lambda \M)$ has canonical singularities.
Suppose the~singularities of the~pair $(X,\lambda \M)$ are non-canonical.
Let us seek for a~contradiction.

Let $B$ be a~center of non-canonical singularities of the~log pair $(X,\lambda \M)$.
Let us create some non-log canonical behaviour using the~center $B$.
In the~following Lemma~\ref{lemma:nlc-hyperplane-section},
we~first treat the~case when $B$ is contained in some special divisor $Y\subset X$, 
so that $(Y,\lambda \M|_Y)$ is not log canonical by the~inversion of adjunction.
As in the~proof of Proposition~\ref{proposition:Fermat-cpi}, we will use Nadel vanishing to~get an estimate of the~possible number of irreducible components of the~non-log canonical locus, and then use Proposition~\ref{proposition:technical} to derive a~contradiction.

\begin{lemma}
\label{lemma:nlc-hyperplane-section}
Let $r,d\geqslant 2$, $n\geqslant 4$ be integers,
let $H$ be a divisor in $|\mathcal{O}_{\mathbb{P}^n}(1)\vert_{X}|$, and set $Y=X\cap X_d$. Assume that $D\sim_{\mathbb{Q}} lH$ is a~ $G$-invariant effective divisor on $X$ for some $l\leqslant (r+1)d$ such that the~divisor $H-(K_X+D)$ is nef. Assume also that $\dim X\geqslant 2$. Then
    \begin{enumerate}
        \item if $D$ does not contain $Y$ in its support, then $(X,D)$ is log canonical;
        \item the~non-log canonical locus of $(X,D)$ is contained in $Y$.
    \end{enumerate}
\end{lemma}

\begin{proof}
As in the proof of Proposition~\ref{proposition:Fermat-cpi}, we may assume that~$l\in \mathbb{N}$.
Suppose that (1) is proved. To prove (2), write
$$
D=t\cdot \frac{l}{d}Y+(1-t)D_0
$$
for some $0\leqslant t\leqslant 1$ and $D_0\sim_{\mathbb{Q}} lH$ such that $Y\not\subseteq\mathrm{Supp}(D_0)$.
Then $(X,D_0)$ is log canonical by~(1).
Hence, every non-log canonical center of $(X,D)$ is a~non-log canonical center of the pair $(X,\frac{l}{d}Y)$.
In particular, the~non-log canonical locus of $(X,D)$ is contained in $Y$. This proves (2).

Now, let us prove (1). Suppose that $Y\not\subseteq \mathrm{Supp}(D)$, and the log pair $(X,D)$ is not log canonical.
Let us seek for a contradiction.
Let $\lambda=\mathrm{lct}(X,D)$ and $Z=\mathrm{Nklt}(X,\lambda D)$.
By Lemmas \ref{lemma:tie-breaking}~and~\ref{lemma:nNklt-base-locus},
we may further assume that $Z$ is $G$-irreducible, $Z$ is a~disjoint union of its irreducible components,
and $Z=\mathrm{Bs}(\mathcal{L})$ for a $G$-invariant linear system $\mathcal{L}\subset|aH|$, where
$$
a=\left(\frac{r(r+1)}{2}-1\right)d+l-r
$$
satisfies
$$
K_X+D+(n-r+1)H\sim_{\mathbb{Q}} aH.
$$

Let $s$ be the number of irreducible components of $Z$,
and  let $Z_1,\ldots,Z_{s}$ be these~components.
By Nadel vanishing applied to the~multiplier ideal sheaf $\mathcal{J}\big(X,\lambda D\big)$, we have a~surjection
$$
H^0\big(X,\mathcal{O}_X(H)\big)\rightarrow H^0\big(Z,\mathcal{O}_{Z}(H\vert_{Z})\big)=\bigoplus_{i=1}^{s}H^0\big(Z_i,\mathcal{O}_{Z_i}(H\vert_{Z_i})\big),
$$
so $s\leqslant h^0(X,\mathcal{O}_X(H))=n+1$.
But $h^0(Z_i,\mathcal{O}_{Z_i}(H\vert_{Z_{i}}))\geqslant 1$ with strict inequality for $\dim(Z_i)>0$.
Thus, if $s=n+1$, then
$$
n+1=h^0\big(Z,\mathcal{O}_{Z}(H\vert_{Z})\big)=\sum_{i=1}^{s}h^0\big(Z_i,\mathcal{O}_{Z_i}(H\vert_{Z_i})\big),
$$
which gives $\dim(Z)=0$, so that we obtain a contradiction $n+1\geqslant |Z|\geqslant d(n+1)>n+1$ as~$d\geqslant 2$,
since the~length of a~$G$-orbit in $X$ is at least $d(n+1)$. This shows that $s\leqslant n$.

Arguing as in the~proof of Proposition~\ref{proposition:Fermat-cpi},
we see that $Z$ is a~component of the~set-theoretic intersection of some $G$-invariant hypersurfaces of degree at most $md$,
where $m=\lfloor\frac{a}{d}\rfloor<n$.
Set
$$
R=X_d\cap X_{2d}\cap \ldots \cap X_{md}.
$$
Then the~pair $(X,D)$ is not log canonical along $R$.
So, since one has  $Y\not\subseteq \mathrm{Supp}(D)$, we have
$$
\mult_R (D|_Y)\geqslant  \mult_R(D)>1.
$$
On the other hand, as in the~proof of Proposition~\ref{proposition:Fermat-cpi}, we obtain $\mult_R(D|_Y)\leqslant 1$ by Lemma~\ref{lemma:hypertangent}.
The obtained contradiction completes the proof of the~lemma.
\end{proof}

Finally we treat the~general case. Here the~main observation is that if the~center of non-canonical singularities is not contained in the~special divisors, then, as a~consequence of Proposition \ref{proposition:technical}, it~has small dimension, and in this case we can prove the~$G$-birational super-rigidity by a~similar application of the~method of \cite{Zhuang}.

\begin{theorem} \label{thm:superrigid cpi}
Let $d\gg 0$, $r\geqslant 2$ be integers. Assume that $-K_X\sim qH$ where $H$ is the~hyperplane class and $1\leqslant q \leqslant \frac{(r+1)d}{2}$. Then $X$ is $G$-birationally super-rigid.
\end{theorem}

\begin{proof}
Assume the~contrary.
Then, using Noether-Fano inequality \cite{Ch05}, we obtain a~non-canonical log pair $(X,\lambda \M)$ such that $\M$ is a mobile~linear system, and $\lambda \in \mathbb{Q}_{>0}$ such that $K_X+\lambda \M \sim_{\mathbb{Q}} 0$.
Let $B$ be a center of non-canonical singularities of the pair $(X,\lambda \M)$.
Let us seek for a~contradiction.

Observe that the center $B$ is not contained in the~Fermat hypersurface $X_d$.
Indeed, otherwise, by the~inverse of adjunction, the~log pair $(Y,M|_Y)$ is not log canonical, where $Y=X\cap X_d$.
But
$$
dH-(K_Y+\lambda \M|_Y)=-(K_X+\lambda \M)|_Y\sim_{\mathbb{Q}} 0,
$$
which is impossible by Proposition~\ref{proposition:Fermat-cpi}.

We claim that $B$ is not contained in any $T$-invariant hyperplane.
Indeed, suppose $B\subseteq \{x_i=0\}$.
Let $X'=X\cap \{x_i=0\}$, and let $G'$ be the stabilizer subgroup in $G$ of the~hyperplane $\{x_i=0\}$.
Then $G^\prime\cong\mumu_d^n\rtimes \SS_n$, and $X'$ is $G^\prime$-invariant.
Let $\M'=\M|_{X'}$.
Then $(X',\lambda \M')$ is not log canonical along $B$ by the~inverse of adjunction.
But $K_{X'}+\lambda \M'\sim_{\mathbb{Q}} H$, which gives $B\subset X_d$ by Lemma~\ref{lemma:nlc-hyperplane-section}.
However, we already proved that $B\not\subset X_d$.

Now by Remark~\ref{remark:Ziquan}, the~log pair $(X,2\lambda \M)$ is not log canonical along $B$.
Let $\mu$ be the smallest positive rational number such that $B\subset \mathrm{Nklt}(X, \mu \M)$,
and let $Z$ be an~irreducible component of the~locus $\mathrm{Nklt}(X, \mu \M)$ containing $B$.
Then $\mu<2\lambda$, and it follows from \cite[Theorem 3.1]{corti} that
$$
\mult_B (M_1 \cdot M_2)>4/\mu^2,
$$	
for general divisors $M_1$ and $M_2$ in the linear system $\M$.
Moreover, it follows from Lemma \ref{lemma:nNklt-base-locus} that the~subvariety $Z$ is a~component of $\mathrm{Bs}(\mathcal{L})$
for some $G$-invariant linear system $\mathcal{L}\subseteq |\mathcal{O}_X(a)|$, where
$$
a=\left(\frac{r(r+1)}{2}-1\right)d+2q-r
$$
satisfies
$$
K_X+2\lambda \M+(n-r+1)H\sim_{\mathbb{Q}} aH.
$$
Furthermore, we know that $Z$ is not contained in any $T$-invariant hyperplane,
because  $B$ is not contained in any $T$-invariant hyperplane.

Set $m=\lfloor \frac{a}{d}\rfloor$. Then $m<n$.
Now, arguing as in the~proof of Proposition~\ref{proposition:Fermat-hypersurfaces} or Proposition~\ref{proposition:Fermat-cpi},
and using Proposition \ref{proposition:technical}, we see that either the~subvariety $Z$ is a~component of a~set-theoretic intersection of $G$-invariant hypersurfaces of degree at most $md$,
or $\dim(B)\leqslant \dim(Z)\leqslant m-1$.

Suppose that the~subvariety $Z$ is a~component of a~set-theoretic intersection of $G$-invariant hypersurfaces of degree at most $md$.
Let $\Gamma'$ be an irreducible component of $M_1\cdot M_2$ such that
$$
\frac{\mult_Z(\Gamma')}{\mathrm{deg}(\Gamma')}\geqslant \frac{\mult_Z(M_1\cdot M_2)}{\mathrm{deg}(M_1\cdot M_2)}>\frac{1}{\mu^2 \mathrm{deg}(M_1\cdot M_2)}.
$$
Since $Z\not\subset X_d$, we observe that $\Gamma=\Gamma'\cdot X_d$ is a~codimension $2$ cycle on $Y=X\cap X_d$ such that
$$
\frac{\mult_{Z'}(\Gamma)}{\mathrm{deg}(\Gamma)}\geqslant \frac{\mult_Z(\Gamma')}{\mathrm{deg}(\Gamma')}>\frac{1}{\mu^2 \mathrm{deg}(M_1\cdot M_2)},
$$
where $Z'=X_d\cap \ldots \cap X_{md} \subseteq Z$. On the other hand, let $H_1$ and $H_2$ be general divisors in $|H|$.
Then, as in the~proof of Proposition~\ref{proposition:Fermat-cpi}, we get
$$
\frac{\mult_{Z'}(\Gamma)}{\mathrm{deg}(\Gamma)}\leqslant \frac{1}{\mathrm{deg} Z'} \prod_{k=r+3}^m kd =\frac{1}{d^2 (r+1)(r+2)\mathrm{deg} (H_1\cdot H_2)}<\frac{1}{\mu^2 \mathrm{deg}(M_1\cdot M_2)},
$$
by Lemma \ref{lemma:hypertangent} applied to the~divisors $Y\cdot X_{kd}$, for  $r+1\leqslant k\leqslant m$ on the~complete intersection $Y$, where the~last inequality follows from $2q\leqslant (r+1)d$. This gives us a contradiction.

Thus, we have $\dim B\leqslant \dim Z\leqslant m-1$.
Let $P$ be a~sufficiently general point in the center $B$,
and let~$Y$ be a~general codimension $m$ linear section of the~complete intersection $X$ containing~$P$.
Set $\M_Y=\M|_Y$.
Then the~pair $(Y,\lambda \M_Y)$ is not log canonical at $P$ by the~inverse of adjunction,
but the~pair $(Y,2\lambda \M_Y)$ is log canonical in a~punctured neighbourhood of the~point $P$.
Note that
$$
K_Y+2\lambda \M_Y\sim_{\mathbb{Q}} (m+q)H.
$$
Using \cite[Corollary~1.8]{Zhuang} and the~lower bound in \cite[Paragraph~56]{Kollar} for the~number of lattice~points,
we obtain the following inequality:
\begin{equation}
\label{equation:lower-bound}
h^0\big(Y,\mathcal{O}_Y((m+q)H\vert_{Y})\big)>\frac{1}{n-m-r}\binom{2(n-m-r)}{n-m-r} \geqslant \frac{1}{(n-m-r)^2} 4^{n-m-r}
\end{equation}
On the~other hand, since $Y$ is a~complete intersection in $\mathbb{P}^{n-m}$, we also have
\begin{equation}
\label{equation:upper-bound}
h^0\big(Y,\mathcal{O}_Y((m+q)H\vert_{Y})\big)\leqslant h^0(\mathbb{P}^{n-m},\mathcal{O}_{\mathbb{P}^{n-m}}(m+q))=\binom{n+q}{m+q}<2^{n+q}.
\end{equation}
Recall that $m=\lfloor \frac{a}{d}\rfloor \leqslant \frac{r(r+1)}{2}+2r$ is bounded by a~constant (that does not depend on $d$),
and, by assumption, we have
$$
n=\left(\frac{r(r+1)}{2}-1\right)d+q-1\geqslant \frac{1}{r+1}\left(\frac{r(r+1)}{2}-1\right)q+q-1\geqslant \frac{5}{3}q-1.
$$
Therefore, using \eqref{equation:upper-bound}, we obtain
$$
h^0\big(Y,\mathcal{O}_Y((m+q)H)\big)<2^{1.6n+1}<\frac{1}{(n-m-r)^2} 4^{n-m-r}
$$
when $n\gg 0$, which is equivalent to $d\gg 0$. This contradicts to \eqref{equation:lower-bound}. The proof is complete.
\end{proof}

\end{document}